\documentclass[a4paper,10pt,reqno]{amsart}
\usepackage{amssymb,latexsym,bbm,amsmath,amsfonts}
\usepackage{amsthm}
\usepackage[mathscr]{eucal}
\usepackage{rotating}
\usepackage{multirow}
\usepackage{slashbox}

\numberwithin{equation}{section}
\newtheorem{theorem}{Theorem}
\newtheorem{lemma}{Lemma}

\theoremstyle{definition}

\newtheorem{assumption}[theorem]{Assumption}

\theoremstyle{remark}
\newtheorem{remark}{Remark}

\begin{document}

\title[SSBE preserves asymptotic behaviour in SDEs]
{On the Dynamic Consistency of the Split Step Method for Classifying the Asymptotic Behaviour of
Globally Stable Differential Equations perturbed by State--independent Stochastic terms}

\author{John A. D. Appleby}
\address{Edgeworth Centre for Financial Mathematics, School of Mathematical
Sciences, Dublin City University, Glasnevin, Dublin 9, Ireland}
\email{john.appleby@dcu.ie} \urladdr{webpages.dcu.ie/\textasciitilde
applebyj}

\author{Jian Cheng}
\address{Edgeworth Centre for Financial Mathematics, School of Mathematical Sciences,
Dublin City University, Glasnevin, Dublin 9, Ireland.}
\email{jian.cheng2@mail.dcu.ie}

\author{Alexandra Rodkina}
\address{Department of Mathematics and Computer Science, Mona Campus, The University of the West Indies, Mona, Kingston 7, Jamaica}
\email{alexandra.rodkina@uwimona.edu.jm}


\thanks{The first and second authors was partially funded by the Science Foundation Ireland
grant 07/MI/008 ``Edgeworth Centre for Financial Mathematics''.}
\date{18 January 2013}

\begin{abstract}
In this paper we classify the pathwise asymptotic behaviour of the discretisation of a general autonomous scalar differential
equation which has a unique and globally stable equilibrium.
The underlying continuous equation  is subjected to
a stochastic perturbation whose intensity is state--independent. In the main result, it is shown that when
the split--step--method is applied to the resulting stochastic differential equation, and the stochastic intensity is decreasing,
the solutions of the discretised equation inherit the asymptotic behaviour of the continuous equation, regardless of whether
the continuous equation has stable, bounded but unstable, or unbounded solutions,
provided the step size is chosen sufficiently small.
\end{abstract}

\maketitle

\section{Introduction}
In this paper the asymptotic behaviour of certain discretisations of
perturbed nonlinear ordinary and stochastic differential equations
is considered. We consider the perturbed stochastic differential equation 
\begin{equation} \label{eq.sdeintro}
dX(t)=-f(X(t))\,dt + \sigma(t)\,dB(t), \quad t\geq 0.
\end{equation}
The equation is finite--dimensional, with
$f:\mathbb{R}^d\to\mathbb{R}^d$,
$:[0,\infty)\to\mathbb{R}^{d\times r}$ and $B$ being an
$r$--dimensional standard Brownian motion. We presume that $f$ and $\sigma$ are sufficiently smooth to ensure the existence of
unique solutions. The appropriate conditions are that $f$ is locally
Lipschitz continuous and that $\sigma$ is continuous.
Throughout we assume that the unperturbed differential
equation
\begin{equation} \label{eq.odeunpert}
y'(t)=-f(y(t)), \quad t\geq 0
\end{equation}
has a unique equilibrium which is translated to zero:
\begin{equation} \label{eq.f0unique}
f(x)=0 \quad\text{if and only if } x=0.
\end{equation}
This equilibrium is globally stable by imposing the dissipative condition
\begin{equation} \label{eq.fdiss}
\langle x, f(x)\rangle > 0 , \quad \text{for all $x\neq 0$}.
\end{equation}
Existence of a continuous solution of \eqref{eq.odeunpert} is guaranteed by assuming that
\begin{equation} \label{eq.fcns}
f\in C(\mathbb{R}^d;\mathbb{R}^d)
\end{equation}
The assumptions \eqref{eq.f0unique}, \eqref{eq.fdiss} and \eqref{eq.fcns} imply that all continuous solutions $y$ of \eqref{eq.odeunpert} obey $y(t)\to 0$ as $t\to\infty$.

The purpose of this paper is to examine whether we can mimick the asymptotic behaviour of the solutions of
\eqref{eq.sdeintro} under discretisation. 
This should be achieved using only the conditions required to ensure
stability, boundedness or unboundedness in the continuous--time case. A particular challenge is to
perform a successful discretisation even in the case when the
function $f$ is not globally linearly bounded, and with a uniform
mesh size $h>0$ if possible. As already discussed in the previous Chapter, it is known that for such highly nonlinear equations that explicit methods are unlikely to preserve the long run behaviour of solutions; see examples in \cite{MSH:2002} and \cite{HMY:2007}. It has been shown in the deterministic case by Stuart Humphries and for stochastic differential equations
that implicit methods are very useful for achieving such results. For this reason, we have adopted the split--step backward Euler method (SSBE) developed in \cite{HMS:2002,MSH:2002}. This method reduces to the standard backward Euler method for deterministic differential equations~\cite{DekkerVerwer:1984,HairerWanner:1996}. In this work, we demonstrate that the split step backward Euler method for SDEs, which was introduced by Mao, Higham and Stuart, and by Mattingly, Stuart and Higham achieves these ends.

The results in this Chapter extend and improve those presented in~\cite{JAJCAR:2010atlanta}, in which a scalar equation
with a monotone increasing $f$ was considered. A classification of the solutions of scalar linear stochastic differential
equations in continuous time was presented in \cite{JAJCAR:2011dresden}.

\section{The Equation}

\subsection{Set--up of the problem}
Suppose that $(\Omega,\mathcal{F},\mathbb{P})$ is a complete probability space. Suppose that $\xi$ is a stochastic sequence in $\mathbb{R}^r$ with the following property:
\begin{assumption} \label{ass.normal}
$\xi=\{\xi(n):n\geq 1\}$ is a sequence of $r$--dimensional independent and identically distributed Gaussian vectors. Moreover, with the notation
$\xi^{(j)}(n)=\langle \xi(n), e_j\rangle$ for $j=1,\ldots,r$, we assume each of the Gaussian random variables $\xi^{j}(n)$ has zero mean and unit variance, and that $\xi^{(j)}(n)$, $j=1,\ldots,r$ are mutually independent for each $n$.
\end{assumption}
This sequence generates a natural filtration $\mathcal{F}(n):=\sigma\{\xi(j):1\leq j\leq n\}$.
In what follows we denote by $\Phi:\mathbb{R}\to \mathbb{R}$ the distribution of a standard normal random variable 
 i.e.,
\begin{equation} \label{def.Phi}
\Phi(x)=\frac{1}{\sqrt{2\pi}}\int_{-\infty}^x e^{-u^2/2}\,du, \quad x\in\mathbb{R}.
\end{equation}
We interpret $\Phi(-\infty)=0$ and $\Phi(\infty)=1$.

We often suppose that $f$ obeys 
\begin{equation} \label{eq.fglobalunperturbed}
f\in C(\mathbb{R}^d;\mathbb{R}^d); \quad \langle x,f(x)\rangle >0, \quad \text{for all $x\in \mathbb{R}^d\setminus\{0\}$}, \quad f(0)=0.
\end{equation}
\begin{remark}\label{re.dissi}
If $f:\mathbb{R}^d\to\mathbb{R}^d$ obeys \eqref{eq.fglobalunperturbed}, the equilibrium at zero of the
unperturbed equation is unique. Suppose to the contrary that there
is $x^\ast\neq 0$ such that $f(x^\ast)=0$. Then $0<\langle
x^\ast,f(x^\ast)\rangle=\langle x^\ast,0\rangle =0$, a
contradiction.
\end{remark}
Suppose also that
\begin{equation} \label{eq.Sigcns}
\Sigma\in C([0,\infty);\mathbb{R}^{d\times r}).
\end{equation}
We consider uniform discretisation of the stochastic differential equation
\begin{equation} \label{eq.sdefinite}
dX(t)=-f(X(t))\,dt + \Sigma(t)\,dB(t), \quad t\geq 0; \quad X(0)=\zeta\in\mathbb{R}^d.
\end{equation}
If, for example, we wish to guarantee the existence of a unique strong solution of \eqref{eq.sdefinite}, we may assume that
$f$ is locally Lipschitz continuous on $\mathbb{R}^d$ or satisfies a global one--sided Lipschitz condition.

However, if one wants only to assure the existence of a solution, the continuity of $f$ and $\sigma$ guarantee the existence of a local solution.
Moreover, the second part of condition \eqref{eq.fglobalunperturbed} guarantees that any such continuous solution does not
explode in finite time almost surely, so we have global existence of the solution.
 Local existence and uniqueness is standard from e.g., \cite{Mao1};
a proof of non--explosion and global existence is given in \cite{JAJGAR:2009}.

\subsection{Construction of the discretisation and existence and uniqueness of its solutions}
We propose to discretise the strong solution $X$ of \eqref{eq.sdefinite} as follows. Let $h>0$, and let
$\sigma_h:\mathbb{N}_0\to\mathbb{R}^{d\times r}$ be a $d\times r$--matrix valued sequence with real entries.
Let $\xi$ be the sequence defined by Assumption~\ref{ass.normal}. Consider the system of stochastic difference equations described by
\begin{subequations}\label{eq.splitstep}
\begin{gather} \label{eq.ssX0}
X_h(0)=\zeta; \\
\label{eq.SSXast}
X_h^\star(n)=X_h(n)-hf(X_h^\star(n)), \quad n\geq 0; \\
\label{eq.SSupdate}
X_h(n+1)=X_h^\star(n)+\sqrt{h}\sigma_h(n)\xi(n+1), \quad n\geq 0.
\end{gather}
\end{subequations}
\eqref{eq.SSXast}, \eqref{eq.SSupdate} with the initial condition \eqref{eq.ssX0} is the so--called \emph{split--step method} for
discretising the stochastic differential equation
\eqref{eq.sdefinite}. This makes sense if we presume that $\sigma_h(n)=\Sigma(nh)$ for $n\geq 0$, where $\Sigma$
is the diffusion coefficient in \eqref{eq.sdefinite}.

\subsection{Existence and uniqueness of solutions of split--step scheme}
We \emph{assume at first} that \eqref{eq.splitstep} has at least one well--defined solution
This is assured by the following deterministic--- and potentially $h$--dependent--- condition on $f$.
\begin{assumption} \label{ass.existl}
For every $x\in\mathbb{R}^d$ there exists $x^\star\in \mathbb{R}^d$ such that
\begin{equation} \label{eq.splitimplicit}
x^\star=x-hf(x^\star).
\end{equation}
\end{assumption}
In this situation, we say that \eqref{eq.splitstep} has a solution if there is a pair of processes $(X_h,X^\star_h)$ which obey
\eqref{eq.splitstep}. Such a solution will automatically be global (i.e, defined for all $n\geq 0$): there is no
possibility of finite time explosion, because each member of the sequence $\xi$ is a.s. finite.
Such a solution will be adapted to the natural filtration generated by $\xi$.
\begin{remark} \label{rem.scalarexist}
In the scalar case $(d=1)$, and $f$ obeys \eqref{eq.fglobalunperturbed}, then Assumption~\ref{ass.existl} is satisfied.
\end{remark}
\begin{proof}
Consider for each $x\in\mathbb{R}$ the function $G_x:\mathbb{R}\to\mathbb{R}$
\[
G_x(y)=y-x+hf(y), \quad y\in\mathbb{R}.
\]
Notice that the continuity of $f$ ensures that $G_x$ is continuous. Then $G_x(0)=-x$ and $G_x(x)=hf(x)$.
Therefore by \eqref{eq.fglobalunperturbed}, $G_x(0)G_x(x)=-hxf(x)<0$ for $x\neq 0$, so that there is a solution
$x^\star$ of \eqref{eq.splitimplicit} between $0$ and $x$ for every $x\neq 0$. In the case when $x=0$, we have $yG_0(y)=y^2 + yhf(y)>0$
for $y\neq 0$ and $G_0(0)=0$. Thus $0$ is the only solution of \eqref{eq.splitimplicit} in the case when $x=0$.
\end{proof}
Conditions can be imposed on $f$ which guarantee that there is a unique solution of \eqref{eq.splitstep}.
These include $f$ obeying the so--called one--sided (global) Lipschitz condition
\[
(f(x)-f(y))(x-y)\leq \mu (x-y)^2, \quad \text{for all $x,y\in\mathbb{R}$}
\]
and some $\mu\in\mathbb{R}$. This condition guarantees the existence of a unique solution of \eqref{eq.splitimplicit}
provided the step size $h$ is chosen to be sufficiently small. Although this is weaker than requesting that $f$ satisfy a
global Lipschitz condition, it places a restriction on $f$ on all $\mathbb{R}$, and still excludes some functions $f$ which
grow faster than polynomially as $|x|\to\infty$.

In this chapter, we do not worry about the uniqueness of the solution of \eqref{eq.splitstep}.
Instead, we show that \emph{all} solutions of the equation will have the correct asymptotic behaviour.
This is in the spirit of generalised dynamical systems considered by Stuart and Humphries~\cite{SH:1998}.
This enables us to impose a weaker regularity condition on $f$ and to therefore consider a wider class of
functions $f$ than are covered by the one--sided Lipschitz condition. But if uniqueness of the solution of
\eqref{eq.splitstep} is required, we are still free to impose extra conditions on $f$.

\subsection{Mean reversion of split--step method under \eqref{eq.fglobalunperturbed}}
Before proceeding, it is worthwhile to note that the first, ``deterministic'' equation in the split--step method (namely \eqref{eq.SSXast}) forces the intermediate estimate $X^\star_h$ to always be closer to the equilibrium than $X_h$.
\begin{lemma} \label{lem.Xastcontract}
Suppose $(X_h,X^\star_h)$ is a solution of \eqref{eq.splitstep} and that $f$ obeys \eqref{eq.fglobalunperturbed}. Then for each $n\in \mathbb{N}$,
\begin{equation*}
0<\|X_h^\star(n)\|<\|X_h(n)\|  \text{ if $\|X_h(n)\|>0$, and $X_h^\star(n)=0$ if and only if $X_h(n)=0$}.
\end{equation*}
\end{lemma}
\begin{proof}
To prove part (a), suppose first that $\|X_h(n)\|>0$. Notice from \eqref{eq.SSXast} that $X_h^\star(n)=0$ implies that $X_h(n)=0$, so we have
$\|X_h^\star(n)\|>0$. By taking the innerproduct with $X_h^\star(n)$ on each side of \eqref{eq.SSXast}, and using the second statement in \eqref{eq.fglobalunperturbed} we get
\[
\|X_h^\star(n)\|^2=\langle X_h(n), X_h^\star(n)\rangle -h\langle f(X_h^\star(n)),X_h^\star(n)\rangle< \langle X_h(n), X_h^\star(n)\rangle.
\]
Applying the Cauchy--Schwartz inequality to the rightmost inequality, this implies that $\|X_h^\star(n)\|^2<\|X_h(n)\|\|X_h^\star(n)\|$, as required.

We have already seen that $X_h^\star(n)=0$ implies that $X_h(n)=0$. To prove the converse, let $X_h(n)=0$ and suppose that $\|X_h^\star(n)\|>0$.
From \eqref{eq.SSXast} we have $X_h^\star(n)=-hf(X_h^\star(n))$, so taking the innerproduct as before and using \eqref{eq.fglobalunperturbed} yields $0<\|X_h^\star(n)\|^2=-h\langle f(X_h^\star(n)),X_h^\star(n)\rangle<0$, a contradiction.
\end{proof}

\section{Statement and Discussion of Main Results}
\subsection{Affine equations}
Before discussing the asymptotic behaviour of solutions of \eqref{eq.splitstep}, it is fruitful to first understand the asymptotic behaviour of the
$d$--dimensional sequence $U_h=\{U_h(n):n\geq 1\}$ defined by
\begin{equation} \label{def.U}
U_h(n+1)=\sqrt{h}\sigma_h(n)\xi(n+1), \quad n\geq 0
\end{equation}
Define
\begin{equation} \label{def.Seps}
S_h(\epsilon)= \sum_{n=0}^\infty \left\{ 1-\Phi\left(\frac{\epsilon}{\|\sigma_h(n)\|_F}  \right)  \right\}.
\end{equation}
Notice that $S_h(\epsilon)$ is monotone in $\epsilon>0$. Therefore, there are only three possible types of behaviour for $S$, for a given $\sigma_h$,
namely: (i) $S_h(\epsilon)<+\infty$ for all $\epsilon>0$; (ii) $S_h(\epsilon)=+\infty$ for all $\epsilon>0$; and (iii) $S_h(\epsilon)<+\infty$
for all $\epsilon>\epsilon'>0$ and $S_h(\epsilon)=+\infty$ for all $\epsilon<\epsilon'$. Due to this trichotomy, it can be seen
that the following result enables the long--run pathwise behaviour of $U_h(n)$ to be classified in terms of $S_h$.
\begin{lemma} \label{lemma.Uasy}
Let
$\xi=\{\xi(n)\in\mathbb{R}^r:n\in\mathbb{N}\}$ be a sequence of
random vectors obeying Assumption~\ref{ass.normal}.
Let $U_h$ be given by \eqref{def.U}, and $S_h(\epsilon)$ be defined by \eqref{def.Seps}.
\begin{itemize}
\item[(A)] If $S_h(\epsilon)<\infty$ for all $\epsilon>0$, then
\begin{equation} \label{eq.Uto0}
\lim_{n\to\infty} U_h(n)=0, \quad \text{a.s.}
\end{equation}
\item[(B)] If $S_h(\epsilon)=+\infty$ for all $\epsilon>0$, then
\begin{equation} \label{eq.Utoinfty}
\limsup_{n\to\infty} \|U_h(n)\|=+\infty, \quad\text{a.s.}
\end{equation}
\item[(C)] If $S_h(\epsilon)<+\infty$ for all $\epsilon>\epsilon'$, and
$S_h(\epsilon)=+\infty$ for all $\epsilon<\epsilon'$, then there exist deterministic $0<c_1\leq c_2<+\infty$ such that
\begin{equation}  \label{eq.Ubounds}
 c_1\leq \limsup_{n\to\infty} \|U_h(n)\|\leq c_2<+\infty, \quad\text{a.s.}
 \end{equation}
\end{itemize}
\end{lemma}
This result enables us to classify the asymptotic behaviour of the discretisation of the $d$--dimensional affine
stochastic differential equation
\begin{equation} \label{eq.linearSDE}
dY(t)=AY(t)\,dt + \Sigma(t)\,dB(t), \quad t\geq 0; \quad X(0)=\zeta,
\end{equation}
where $A$ is a $d\times d$ matrix with real entries. We assume
that all solutions of the underlying deterministic differential equation
\begin{equation} \label{eq.linearode}
y'(t)=Ay(t), \quad t>0, \quad x(0)=\zeta
\end{equation}
obey $y(t)\to 0$ as $t\to\infty$. This means that
\begin{equation} \label{eq.linstab}
\text{$\text{Re}(\lambda)<0$ for all eigenvalues $\lambda$ of $A$}.
\end{equation}
Let $c_A$ be the characteristic polynomial of $A$, so that $c_A(\lambda)=\det(\lambda I_d -A)$. By \eqref{eq.linstab}, it follows that
there are no positive real solutions of the characteristic equation $c_A(\lambda)=0$. In particular, $c_A(1/h)\neq 0$ for every $h>0$, so
we have that $\det(I-hA)\neq 0$ and therefore the matrix $C(h)$ given by
\begin{equation} \label{def.Ch}
C(h)=(I-hA)^{-1}
\end{equation}
is well--defined. Therefore, there is a unique solution of the split--step scheme
\begin{subequations} \label{eq.splitsteplin}
\begin{align}
\label{eq.linic}
Y_h(0)&=\zeta, \\
\label{eq.linss}
Y^\star_h(n)&=Y_h(n)+hAY^\star_h(n), \quad n\geq 0, \\
\label{eq.linupdate}
Y_h(n+1)&=Y_h^\star(n)+\sqrt{h}\sigma_h(n)\xi(n+1), \quad n\geq 0.
\end{align}
\end{subequations}
which is equivalent to
\[
Y_h(n+1)=C(h)Y_h(n)+\sqrt{h}\sigma_h(n)\xi(n+1), \quad n\geq 0; \quad Y_h(0)=\zeta.
\]
The asymptotic behaviour of $Y_h$ can now be given.
\begin{theorem}  \label{thm.linsplit}
Suppose that $A\in \mathbb{R}^{d\times d}$ obeys \eqref{eq.linstab}. Let $\xi=\{\xi(n)\in\mathbb{R}^r:n\in\mathbb{N}\}$ be a sequence of
random vectors obeying Assumption~\ref{ass.normal}. Let $S_h(\epsilon)$ be defined by \eqref{def.Seps}, and
$(Y_h,Y_h^\star)$ be the unique solution of \eqref{eq.splitsteplin}.
\begin{itemize}
\item[(A)] If $S_h(\epsilon)<+\infty$ for every $\epsilon>0$, then $Y_h(n)\to 0$ as $n\to\infty$ a.s.
\item[(B)] If there exists $\epsilon'>0$ such that $S_h(\epsilon)<+\infty$ for all $\epsilon>\epsilon'$
and $S_h(\epsilon)=+\infty$ for all $\epsilon<\epsilon'$, then there exist deterministic $0<c_3\leq c_4<+\infty$ such that
\[
c_3\leq \limsup_{n\to\infty} \|Y_h(n)\|\leq c_4, \quad \text{a.s.}
\]
and 
\[
\liminf_{n\to\infty} \|Y_h(n)\|=0, \quad \lim_{n\to\infty} \frac{1}{n}\sum_{j=1}^n \|Y_h(j)\|^2=0, \quad\text{a.s.}
\]
\item[(C)] If $S_h(\epsilon)=+\infty$ for all $\epsilon>0$, then $\limsup_{n\to\infty} \|Y_h(n)\|=+\infty$ a.s.
\end{itemize}
\end{theorem}

\subsection{Nonlinear equation}
We now discuss the asymptotic behaviour of solutions of \eqref{eq.splitstep}.
We first show that $X_h$ has a zero limit in the case when $\sigma_h$ is square summable, without
placing any condition on $f$ stronger than \eqref{eq.fglobalunperturbed}.
This is a direct analogue of results available in continuous time.
\begin{theorem} \label{thm.sig2l1Xto0}
Suppose that $(X_h,X^\star_h)$ is a solution of
\eqref{eq.splitstep}. Suppose $f$ obeys
\eqref{eq.fglobalunperturbed}, $\sigma_h\in \ell^2(\mathbb{N},\mathbb{R})$, and that the
sequence $\xi$ obeys Assumption~\ref{ass.normal}. Then $\lim_{n\to \infty} X_h(n)=0$, a.s.
\end{theorem}

We show that when $U_h$ is unbounded, so is $\|X_h\|$, and also that if $U_h$ is bounded, $\|X_h\|$ is bounded away
from zero by a deterministic constant.
\begin{theorem} \label{theorem.XunboundedXboundedbelow}
Suppose that $(X_h,X^\star_h)$ is a solution of \eqref{eq.splitstep}.
Suppose that $f$ obeys \eqref{eq.fglobalunperturbed} and that the sequence $\xi$ obeys Assumption~\ref{ass.normal}.
Let  $S_h(\epsilon)$ be defined by \eqref{def.Seps}.
\begin{itemize}
\item[(A)]
If $S_h(\epsilon)=+\infty$ for every $\epsilon>0$, then 
\[
\limsup_{n\to\infty} \|X_h(n)\|=+\infty, \quad \text{a.s.}
\]
\item[(B)]
If $S_h(\epsilon)<+\infty$ for all $\epsilon>\epsilon'$ and $S_h(\epsilon)=+\infty$ for all $\epsilon<\epsilon'$, then
\[
\limsup_{n\to\infty} \|X_h(n)\|\geq \frac{c_1}{2}, \quad \text{a.s.},
\]
where $c_1$ is defined by \eqref{eq.Ubounds}.
\end{itemize}
\end{theorem}

\begin{theorem} \label{thm:Xlim0}
Suppose that $(X_h,X^\star_h)$ is a solution of
\eqref{eq.splitstep}. Suppose that $f$ obeys
\eqref{eq.fglobalunperturbed} 
and that the sequence $\xi$ obeys Assumption~\ref{ass.normal}. Let $S_h(\epsilon)$ be defined by \eqref{def.Seps}.
\begin{itemize}
\item[(i)]If $S_h(\epsilon)<+\infty$ for every $\epsilon>0$, then
\[
\{\lim_{n\to \infty}\|X_h(n)\|=0\}\cup\{\lim_{n\to\infty} \|X_h(n)\|=+\infty\} \text{ is an a.s. event}.
\]
\item[(ii)] If $\lim_{n\to\infty} X_h(n)=0$ with positive probability, then
$S_h(\epsilon)<+\infty$ for every $\epsilon>0$.
\end{itemize}
\end{theorem}

%

Under an additional mean--reverting condition on $f$, we can characterise the conditions on $\sigma_h$ under which solutions of
\eqref{eq.splitstep} tend to zero.
\begin{theorem} \label{thm:Xlim0faway0}
Suppose that $(X_h,X^\star_h)$ is a solution of
\eqref{eq.splitstep}. Suppose that $f$ obeys
\eqref{eq.fglobalunperturbed} and
\begin{equation} \label{eq.xfxunifpos}
\liminf_{y\to\infty} \inf_{\|x\|=y} \langle x,f(x)\rangle =:\phi>0.
\end{equation}
and that the sequence $\xi$ obeys Assumption~\ref{ass.normal}.
\begin{itemize}
\item[(A)] $S_h(\epsilon)$ defined by \eqref{def.Seps} obeys $S_h(\epsilon)<+\infty$ for every
$\epsilon>0$;
\item[(B)] $\lim_{n\to \infty} X_h(n)=0$ a.s. for all $\zeta\in\mathbb{R}^d$;
\item[(C)] $\lim_{n\to \infty} X_h(n)=0$ with positive probability for some $\zeta\in\mathbb{R}^d$;
\end{itemize}
\end{theorem}

Furthermore, in the scalar case, we can characterise the stability of the equilibrium without requiring to assume \eqref{eq.xfxunifpos}. In fact, it suffices to just assume that $f$ obeys \eqref{eq.fglobalunperturbed}.
\begin{theorem}\label{theorem.xto0}
Suppose that $(X_h,X^{\star}_h)$ is a solution of \eqref{eq.splitstep}. Suppose that $f$ obeys \eqref{eq.fglobalunperturbed} and $S_h(\epsilon)=\sum^{\infty}_{n=0}\left\{1-\Phi(\epsilon/|\sigma_n(n)|)\right\}<+\infty$ for all $\epsilon>0$. Then
\[
\lim_{n\to \infty}X_h(n,\zeta)=0  \quad \text{a.s. for all $\zeta\in \mathbb{R}$}. 
\]
\end{theorem}

The next result enables us to completely classify the asymptotic behaviour of the solutions of \eqref{eq.splitstep}.
In order to do so, we must strengthen once again the mean--reverting hypothesis on $f$.

%

\begin{theorem} \label{lemma.Xboundedabove}
Suppose that $(X_h,X^\star_h)$ is a solution of \eqref{eq.splitstep}.
Suppose that $f$ obeys \eqref{eq.fglobalunperturbed} and
\begin{equation} \label{eq.fboundedbelow}
\liminf_{y\to\infty} \inf_{\|x\|=y} \frac{\langle x,f(x)\rangle}{\|x\|}=+\infty,
\end{equation}
and that the sequence $\xi$ obeys Assumption~\ref{ass.normal}. Let $S_h(\epsilon)$ be defined by \eqref{def.Seps}.
\begin{itemize}
\item[(A)]  If $S_h(\epsilon)<+\infty$ for all $\epsilon>0$, then $\lim_{n\to\infty} X_h(n)=0$ a.s.
\item[(B)] If $S_h(\epsilon)<+\infty$ for all $\epsilon>\epsilon'$ and $S_h(\epsilon)=+\infty$ for all $\epsilon<\epsilon'$, then
there exists deterministic $0<c_3\leq c_4<+\infty$ such that
\[
c_3<\limsup_{n\to\infty} \|X_h(n)\|\leq c_4, \quad \text{a.s.}
\]
and 
\[
\liminf_{n\to\infty} \|X_h(n)\|=0, \quad\text{a.s.}
\]
\item[(C)]  If $S_h(\epsilon)=+\infty$ for all $\epsilon>0$, then $\limsup_{n\to\infty} \|X_h(n)\|=+\infty$ a.s.
\end{itemize}
\end{theorem}

This necessary and sufficient condition on $S_h(\epsilon)$ is difficult to evaluate directly, because we do not know $\Phi$ in its closed form. However we can show that $S_h(\epsilon)$ is finite or infinite according as to whether the sum 
\begin{equation} \label{def.sh'epsilon}
S_h'(\epsilon)=\sum_{n=0}^{\infty}\|\sigma_h(n)\|_F\exp\left(-\frac{\epsilon^2}{2}\frac{1}{\|\sigma_h(n)\|^2_F}\right)
\end{equation}
is finite or infinite, we interpret the summand to be zero when $\|\sigma_h(n)\|_F=0$. Therefore we establish the following Lemmata 
which enables us to obtain all the above results with $S_h'(\epsilon)$ in place of $S_h(\epsilon)$.
\begin{lemma}\label{lemma.sepss'epsdisc}
$S_h(\epsilon)$ given by \eqref{def.Seps} is finite if and only if
$S_h'(\epsilon)$ given by \eqref{def.sh'epsilon} is finite.
\end{lemma}
\begin{proof}
We note by e.g., \cite[Problem 2.9.22]{K&S}, we have 
\begin{equation} \label{eq.millsasy}
\lim_{x\to\infty}\frac{1-\Phi(x)}{x^{-1}e^{-x^2/2}}=\frac{1}{\sqrt{2\pi}}.
\end{equation}
If $S_h(\epsilon)$ is finite, then $1-\Phi(\epsilon/\|\sigma_h(n)\|_F)\to 0$
as $n\to\infty$. This implies $\epsilon/\|\sigma_h(n)\|_F\to\infty$ as
$n\to\infty$. Therefore by \eqref{eq.millsasy}, we have
\begin{equation} \label{eq.millsthetanmult}
\lim_{n\to\infty}\frac{1-\Phi(\epsilon/\|\sigma_h(n)\|_F)}{\|\sigma_h(n)\|_F/\epsilon\cdot
\exp(-\epsilon^2/\{2\|\sigma_h(n)\|_F^2\})}=\frac{1}{\sqrt{2\pi}}.
\end{equation}
Since $\left(1-\Phi(\epsilon/\|\sigma_h(n)\|_F)\right)_{n\geq 1}$ is summable, it
therefore follows that the sequence
\[\left(\|\sigma_h(n)\|_F/\epsilon\cdot
\exp(-\epsilon^2/\{2\|\sigma_h(n)\|_F^2(n)\})\right)_{n\geq 1}
\] is summable, so $S_h'(\epsilon)$ is finite, by definition.

On the other hand, if $S'(\epsilon)$ is finite, and we define
$\phi:[0,\infty)\to \mathbb{R}^d$ by
\[
\phi(x)=\left\{ \begin{array}{cc} x \exp(-1/(2x^2)), & x>0, \\
0, & x=0,
\end{array}
\right.
\]
then we have $\|\sigma_h(n)\|_F\exp(-\epsilon^2/2\|\sigma_h(n)\|_F^2(n))$ is summable, and hence the sequence $(\phi(\|\sigma_h(n)\|_F/\epsilon))_{n\geq 1}$ is summable.
Therefore $\phi(\|\sigma_h(n)\|_F/\epsilon)\to 0$ as $n\to\infty$. Then, as
$\phi$ is continuous and increasing on $[0,\infty)$, we have that
$\|\sigma_h(n)\|_F/\epsilon\to 0$ as $n\to\infty$, or $\epsilon/\|\sigma_h(n)\|_F\to
\infty$ as $n\to\infty$. Therefore \eqref{eq.millsthetanmult} holds, and
thus $(1-\Phi(\epsilon/\|\sigma_h(n)\|_F))_{n\geq 1}$ is summable, which
implies that $S_h(\epsilon)$ is finite, as required.
\end{proof}

\subsection{Connection with continuous results}
To see how these results mimic the asymptotic behaviour of \eqref{eq.sdefinite} and \eqref{eq.linearSDE}, we record
corresponding result for solutions of these equations. To this end, we define
\begin{equation}\label{def.Sepscns}
S(\epsilon)=\sum_{n=0}^\infty \left\{1-\Phi\left(\frac{\epsilon}{\sqrt{\int_n^{n+1} \|\Sigma(t)\|^2_F\,dt}}\right)\right\}
\end{equation}
and for $h>0$
\begin{equation}\label{def.Sepshcns}
S_h^{(c)}(\epsilon)=\sum_{n=0}^\infty
\left\{1-\Phi\left(\frac{\epsilon}{\sqrt{\frac{1}{h}\int_{nh}^{(n+1)h} \|\Sigma(t)\|^2_F\,dt}}\right)\right\}.
\end{equation}
Perusal of results in ~\cite{JAJCAR:2012} show that  
$S(\cdot)$ above can be replaced by $S_h^{c}$. The result therefore is
\begin{theorem}  \label{thm.linsde}
Suppose that $A\in \mathbb{R}^{d\times d}$ obeys \eqref{eq.linstab}. Let $h>0$ and suppose that
$S^{(c)}_h(\epsilon)$ be defined by \eqref{def.Sepshcns}. Let $Y$ be the unique solution of \eqref{eq.linearSDE}.
\begin{itemize}
\item[(A)] If $S_h^{(c)}(\epsilon)<+\infty$ for every $\epsilon>0$, then $Y(t)\to 0$ as $t\to\infty$ a.s.
\item[(B)] If there exists $\epsilon'>0$ such that $S_h^{(c)}(\epsilon)<+\infty$ for all $\epsilon>\epsilon'$
and $S_h^{(c)}(\epsilon)=+\infty$ for all $\epsilon<\epsilon'$, then there exist deterministic $0<c_1\leq c_2<+\infty$ such that
\[
c_1\leq \limsup_{t\to\infty} \|Y(t)\|\leq c_2, \quad \text{a.s.}
\]
and 
\[
\liminf_{t\to\infty} \|Y(t)\|=0, \quad \lim_{t\to\infty} \frac{1}{t}\int_0^t \|Y(s)\|^2\,ds=0, \quad\text{a.s.}
\]
\item[(C)] If $S_h^{(c)}(\epsilon)=+\infty$ for all $\epsilon>0$, then $\limsup_{t\to\infty} \|Y(t)\|=+\infty$ a.s.
\end{itemize}
\end{theorem}
Similarly, we may replace $S$ by $S_h^{c}$ in a result of~\cite{JAJCAR:2012} 
to get
\begin{theorem} \label{thm.cnsclassify}
Suppose that $f$ obeys \eqref{eq.fglobalunperturbed} and \eqref{eq.fboundedbelow}.
Suppose that $X$ is a solution of \eqref{eq.sdefinite}. Let $h>0$ and suppose that
$S^{(c)}_h(\epsilon)$ be defined by \eqref{def.Sepshcns}.
\begin{itemize}
\item[(A)]  If $S_h^{(c)}(\epsilon)<+\infty$ for all $\epsilon>0$, then $\lim_{n\to\infty} X(t)=0$ a.s.
\item[(B)] If $S_h^{(c)}(\epsilon)<+\infty$ for all $\epsilon>\epsilon'$ and $S_h^{(c)}(\epsilon)=+\infty$ for all $\epsilon<\epsilon'$, then
there exist deterministic $0<c_3\leq c_4<+\infty$ such that
\[
c_3<\limsup_{t\to\infty} \|X(t)\|<c_4, \quad \text{a.s.}
\]
and 
\[
\liminf_{t\to\infty} \|X(t)\|=0, \quad\text{a.s.}
\]
\item[(C)]  If $S_h^{(c)}(\epsilon)=+\infty$ for all $\epsilon>0$, then $\limsup_{n\to\infty} \|X(t)\|=\infty$ a.s.
\end{itemize}
\end{theorem}

If we take a uniform step size $h>0$ in a forward Euler--discretisation of \eqref{eq.sdefinite}, this is tantamount to setting
\begin{equation} \label{eq.sighchoice1}
\sigma_h(n)=\Sigma(nh), \quad n\geq 0
\end{equation}
in \eqref{eq.splitstep}. In this case, the continuity of $\Sigma$ ensures for each fixed $n$ that
\[
\lim_{h\to 0} \left\{\frac{1}{h}\int_{nh}^{(n+1)h}
\|\Sigma(s)\|^2_F\,ds - \|\sigma_h(n)\|^2_F\right\} =0,
\]
so it can be seen that the conditions classifying the finiteness $S_h$ and $S_h^{c}$
are in some sense ``close''. We now give some examples where $S_h$ and $S_h^c$ share the same finiteness
properties, and therefore, the asymptotic behaviour of solutions of \eqref{eq.sdefinite} and \eqref{eq.splitstep}
coincide.

In the case when the integral  $\int_a^b \Sigma^2_{ij}(s)\,ds$ can be computed explicitly for any $0\leq a<b<+\infty$ and
$(i,j)\in \{1,\ldots,d\}\times\{1,\ldots,r\}$, it is  reasonable
to approximate the stochastic integral
 \[
 \int_{nh}^{(n+1)h} \Sigma_{ij}(s)\,dB_j(s) \text{ by } \left(\int_{nh}^{(n+1)h} \Sigma^2_{ij}(s)\,ds\right)^{1/2}\xi_j(n+1)
 \]
 where $\xi$ obeys Assumption~\ref{ass.normal}. This is because the two random variables displayed above have
 the same distribution.
 In terms of \eqref{eq.splitstep} (particularly \eqref{eq.SSupdate}) this amounts to
 choosing $\sigma_h$ according to
\begin{equation} \label{eq.sighchoice2}
[\sigma_h(n)]_{ij}=\frac{1}{\sqrt{h}}\left(\int_{nh}^{(n+1)h} \Sigma^2_{ij}(s)\,ds\right)^{1/2}, \quad n\geq 0, 
\quad (i,j)\in \{1,\ldots,d\}\times\{1,\ldots,r\}.
\end{equation}
In this case, it is seen that $S_h(\epsilon)=S_h^{c}(\epsilon)$. Applying Theorems~\ref{lemma.Xboundedabove} and~\ref{thm.cnsclassify},
we immediately have the following result.
\begin{theorem} \label{thm.disccnseqvt2}
Suppose that $f$ obeys \eqref{eq.fglobalunperturbed} and
\eqref{eq.fboundedbelow} and suppose that $\Sigma$ obeys
\eqref{eq.Sigcns}. Assume that the sequence $\xi$ obeys
Assumption~\ref{ass.normal}, and for $h>0$ that $f$ obeys
Assumption~\ref{ass.existl}. Let $X$ be a solution of \eqref{eq.sdefinite}
and $(X_h,X_h^\star)$ is a solution of \eqref{eq.splitstep}. Then
exactly one of the events
\begin{multline*}
\{\omega: \lim_{t\to\infty} X(t,\omega)=0\}, \\
\{\omega: 0<\limsup_{t\to\infty} \|X(t,\omega)\|<+\infty, \,\liminf_{t\to\infty} \|X(t,\omega)\|=0\}, \\
\text{and }
\{\omega: \limsup_{t\to\infty} \|X(t,\omega)\|=+\infty\}
\end{multline*}
is almost sure, and exactly one of the events
\begin{multline*}
\{\omega: \lim_{n\to\infty} X_h(n,\omega)=0\}, \\
\{\omega: 0<\limsup_{n\to\infty} \|X_h(n,\omega)\|<+\infty, \, \liminf_{n\to\infty} \|X_h(n,\omega)\|=0\}, \\
\text{and }
\{\omega: \limsup_{n\to\infty} \|X_h(n,\omega)\|=+\infty\}
\end{multline*}
is almost sure.

If $\sigma_h$ is given by \eqref{eq.sighchoice2}, and $n\mapsto \int_{nh}^{(n+1)h} \Sigma^2_{ij}(s)\,ds$ can be computed exactly
for all $(i,j)\in \{1,\ldots,d\}\times\{1,\ldots,r\}$ and all $n\in \mathbb{N}$, we have the
following equivalences:
\begin{itemize}
\item[(i)]  $\lim_{t\to\infty} X(t)=0$ a.s., if and only if $\lim_{n\to\infty} X_h(n)=0$ a.s.
\item[(ii)]  $\limsup_{t\to\infty} \|X(t)\|\in (0,\infty)$ a.s., if and only if $\limsup_{n\to\infty} \|X_h(n)\|\in (0,\infty)$ a.s.
\item[(iii)]  $\limsup_{t\to\infty} \|X(t)\|=+\infty$ a.s., if and only if $\limsup_{n\to\infty} \|X_h(n)\|=+\infty$ a.s.
\end{itemize}
\end{theorem}

We next consider a situation where finiteness conditions  on $S_h(\epsilon)$ and $S_h^{c}(\epsilon)$ also coincide,
but in which we do not need to have a closed--form expression for $\int_a^b \Sigma_{ij}^2(s)\,ds$. This is the case when
$t\mapsto \|\Sigma(t)\|_F^2$ is decreasing and $\sigma_h(n)=\Sigma(nh)$.
\begin{theorem} \label{thm.disccnseqvt}
Suppose that $f$ obeys \eqref{eq.fglobalunperturbed} and
\eqref{eq.fboundedbelow} and suppose that $\Sigma$ obeys
\eqref{eq.Sigcns}. Assume that the sequence $\xi$ obeys
Assumption~\ref{ass.normal}, and for $h>0$ that $f$ obeys
Assumption~\ref{ass.existl}. Let $X$ be a solution of \eqref{eq.sdefinite}
and $(X_h,X_h^\star)$ is a solution of \eqref{eq.splitstep}. Then
exactly one of the events
\begin{multline*}
\{\omega: \lim_{t\to\infty} X(t,\omega)=0\}, \\
\{\omega: 0<\limsup_{t\to\infty} \|X(t,\omega)\|<+\infty,\,\liminf_{t\to\infty}\|X(t,\omega)\|=0\}, \\
\text{and }
\{\omega: \limsup_{t\to\infty} \|X(t,\omega)\|=+\infty\}
\end{multline*}
is almost sure, and exactly one of the events
\begin{multline*}
\{\omega: \lim_{n\to\infty} X_h(n,\omega)=0\}, \\
\{\omega: 0<\limsup_{n\to\infty} \|X_h(n,\omega)\|<+\infty,\,\liminf_{n\to\infty}\|X_h(n,\omega)\|=0\}, \\
\text{and }
\{\omega: \limsup_{n\to\infty} \|X_h(n,\omega)\|=+\infty\}
\end{multline*}
is almost sure.

If we further suppose that $t\mapsto \|\Sigma(t)\|^2_F$ is non--increasing, and $\sigma_h(n)$ is given by \eqref{eq.sighchoice1},
we have the following equivalences:
\begin{itemize}
\item[(i)]  $\lim_{t\to\infty} X(t)=0$ a.s., if and only if $\lim_{n\to\infty} X_h(n)=0$ a.s.
\item[(ii)]  $\limsup_{t\to\infty} \|X(t)\|\in (0,\infty)$ a.s., if and only if $\limsup_{n\to\infty} \|X_h(n)\|\in (0,\infty)$ a.s.
\item[(iii)]  $\limsup_{t\to\infty} \|X(t)\|=+\infty$ a.s., if and only if $\limsup_{n\to\infty} \|X_h(n)\|=+\infty$ a.s.
\end{itemize}
\end{theorem}
\begin{proof}
Define $\vartheta_h(n)^2=\int_{nh}^{(n+1)h}\|\Sigma(t)\|^2_F\,dt/h$. Since $t\mapsto\|\Sigma(t)\|^2_F$
is non--increasing, for $t\in [nh,(n+1)h]$ we have
$\|\Sigma((n+1)h)\|_F^2\leq \|\Sigma(t)\|_F^2 \leq \|\Sigma(nh)\|_F^2$. Therefore
integrating over $[nh,(n+1)h]$ and using \eqref{eq.sighchoice1} we get $\|\sigma_h(n+1)\|_F\leq
\vartheta_h(n) \leq \|\sigma_h(n)\|_F$.
For $\epsilon>0$, as $\Phi$ is increasing, we have
\[
1-\Phi\left(\frac{\epsilon}{\|\sigma_h(n+1)\|_F}\right)
\leq
1-\Phi\left(\frac{\epsilon}{\|\vartheta_h(n)\|_F}\right)
\leq
1-\Phi\left(\frac{\epsilon}{\|\sigma_h(n)\|_F}\right).
\]
Summing across this inequality and using the definitions \eqref{def.Seps} and \eqref{def.Sepshcns}
we get
\[
S_h(\epsilon)-\left\{ 1-\Phi\left(\frac{\epsilon}{\|\sigma_h(0)\|_F}\right) \right\}
\leq S_h^{(c)}(\epsilon) \leq  S_h(\epsilon).
\]
Therefore, for any $\epsilon>0$, $S_h(\epsilon)$ is finite if and only if $S_h^{(c)}(\epsilon)$ is finite.

We now prove the equivalence (i). Suppose that $\lim_{t\to\infty} X(t)=0$ a.s.
Then, as $S_h^{(c)}(\epsilon)$ must be (i) finite for all $\epsilon>0$; (ii) infinite for all $\epsilon>0$;
or (iii) finite for all $\epsilon>\epsilon'$ and infinite for all $\epsilon<\epsilon'$ for some $\epsilon'>0$, it follows from Theorem~\ref{thm.cnsclassify} that $S_h^{(c)}(\epsilon)<+\infty$ for all $\epsilon>0$.
Therefore, we have that $S_h(\epsilon)<+\infty$ for all $\epsilon>0$. Theorem~\ref{lemma.Xboundedabove}
now implies that $X_h(n)\to 0$ as $n\to\infty$ a.s.

Conversely, suppose that $X_h(n)\to 0$ as $n\to\infty$ a.s. Since
$S_h(\epsilon)$ must be (i) finite for all $\epsilon>0$; (ii) infinite for all $\epsilon>0$;
or (iii) finite for all $\epsilon>\epsilon'$ and infinite for all $\epsilon<\epsilon'$ for some
$\epsilon'>0$, it follows from Theorem~\ref{lemma.Xboundedabove} that $S_h(\epsilon)<+\infty$ for all
$\epsilon>0$. Therefore, we have that $S_h^{(c)}(\epsilon)<+\infty$ for all $\epsilon>0$, and hence by  Theorem~\ref{thm.cnsclassify}, $X(t)\to 0$ as $t\to\infty$ a.s., completing the proof of (i).

The proof of the equivalences (ii) and (iii) are similar, and hence omitted.
\end{proof}
The condition that $S'_h(\epsilon)$ is finite or infinite can be difficult to check. However we can provide a sufficient condition on which each case of $S'_h(\epsilon)$ being finite all the time, sometime finite sometime infinite and infinite all the time is possible according to whether $\lim_{t\to +\infty}\|\sigma_h(n)\|^2_F\log n$ being zero, non-zero and finite, or infinite. Therefore the asymptotic behaviour of the solution of 
\eqref{eq.splitstep} (and indeed \eqref{eq.sdefinite}) can be classified completely.
\begin{lemma}\label{lemma.sigmalog}
Define $\lim_{n\to \infty}\|\sigma_h(n)\|^2_F\log n=L\in[0,\infty]$, then we have the following:
\begin{itemize}
\item[(A)] If $L=0$, then $S'_h(\epsilon)<+\infty$ for all $\epsilon>0$;
\item[(B)] If $L\in(0,+\infty)$, then $S'_h(\epsilon)<+\infty$ for all $\epsilon>\epsilon'$ and $S'_h(\epsilon)=+\infty$ for all $\epsilon<\epsilon'$;
\item[(C)] If $L=+\infty$, then $S'_h(\epsilon)=+\infty$ for all $\epsilon>0$
\end{itemize}
\end{lemma}
\begin{proof}
Notice from e.g., \cite[Problem 2.9.22]{K&S}, $\lim_{x\to\infty}(1-\Phi(x))/(x^{-1}e^{-x^2/2})=1/\sqrt{2\pi}$.
Therefore we have 
\[
\lim_{x\to \infty}\log(1-\Phi(x))+\log x+x^2/2=\log(1/\sqrt{2\pi}),
\]
hence
\[
\lim_{x\to \infty}\frac{\log(1-\Phi(x))}{x^2/2}=-1.
\]
Let $x=\epsilon/\|\sigma_h(n)\|_F\to \infty$ as $n\to \infty$, we have 
\[
\lim_{n\to\infty}\frac{\log(1-\Phi(\epsilon/\|\sigma_h(n)\|_F))}{\epsilon^2/2\|\sigma_h(n)\|_F}=-1.
\]
Moreover, 
\begin{align*}
\lim_{n\to\infty}\frac{\log(1-\Phi(\epsilon/\|\sigma_h(n)\|_F))}{\log n}&=
\lim_{n\to\infty}\frac{\log(1-\Phi(\epsilon/\|\sigma_h(n)\|_F))}{\epsilon^2/2\|\sigma_h(n)\|_F}\cdot\frac{\epsilon^2/2\|\sigma_h(n)\|_F}{\log n}\\
&=-\frac{\epsilon^2}{2}\lim_{n\to \infty}\frac{1}{\|\sigma_h(n)\|_F\log n}
\end{align*}
If $L=0$, then 
\[
\lim_{n\to\infty}\frac{\log(1-\Phi(\epsilon/\|\sigma_h(n)\|_F))}{\log n}\to -\infty.
\]
Therefore there exists an $N(\epsilon)$, such that for $n>N(\epsilon)$ 
\begin{align*}
\log(1-\Phi(\epsilon/\|\sigma_h(n)\|_F))&<-2\log n\\
1-\Phi(\epsilon/\|\sigma_h(n)\|_F)&\leq n^{-2}\to 0 \quad\text{as $n\to \infty$}
\end{align*}
This implies that $S_h(\epsilon)<+\infty$, which implies $S'_h(\epsilon)<+\infty$ by Lemma \ref{lemma.sepss'epsdisc} proving part (A).

If $L\in(0,+\infty)$, we have 
\[
\lim_{n\to\infty}\frac{\log(1-\Phi(\epsilon/\|\sigma_h(n)\|_F))}{\log n}=\frac{-\epsilon^2}{2L}.
\]
Therefore either $\epsilon>\sqrt{2L}$, in which case $\lim_{n\to\infty}1-\Phi(\epsilon/\|\sigma_h(n)\|_F)=0$, hence $S_h(\epsilon)<+\infty$, and $S'_h(\epsilon)<+\infty$. Or $\epsilon<\sqrt{2L}$, in which case $1-\Phi(\epsilon/\|\sigma_h(n)\|_F)$ is not going to zero, hence not summable, therefore $S_h(\epsilon)=+\infty$ which implies $S'_h(\epsilon)=+\infty$.

Finally, if $L=+\infty$, suppose that $S_h(\epsilon)<+\infty$, then
\[
\lim_{n\to\infty}\frac{\log(1-\Phi(\epsilon/\|\sigma_h(n)\|_F))}{\log n}=0.
\]
Then for all $\epsilon>0$, there exists an $N(\epsilon)>0$ such that
\begin{align*}
\frac{\log(1-\Phi(\epsilon/\|\sigma_h(n)\|_F))}{\log n}&>-1/2\\
\log(1-\Phi(\epsilon/\|\sigma_h(n)\|_F))&>-1/2\log n=\log n^{-1/2}\\
1-\Phi(\epsilon/\|\sigma_h(n)\|_F)&>n^{-1/2}\quad \text{for all $n\geq N(\epsilon)$}
\end{align*}
This implies $S_h(\epsilon)=+\infty$, which is a contradiction, hence the required result, completing the proof.
\end{proof}


\section{Preliminary Results}
In this section, we deduce some simple preliminary facts about \eqref{eq.splitstep} contingent on a solution $(X_h,X^\star_h)$ existing.
We also present some results on the asymptotic behaviour of martingales that will be of utility in the sequel.
\subsection{Estimates and representation}
In our next result, we obtain a representation for $\|X_h(n)\|^2$.
\begin{lemma} \label{lemma.repX2}
Suppose $(X_h,X^\star_h)$ is a solution of \eqref{eq.splitstep}. Then
\begin{multline} \label{eq.X2sum}
\|X_h(n)\|^2=\|X_h(0)\|^2 - 2\sum_{i=1}^n h\langle f(X_h^\star(i-1)), X_h^\star(i-1)\rangle +\sum_{i=1}^n h\|\sigma_h(i-1)\xi(i)\|^2
\\- \sum_{i=1}^n h^2\|f(X_h^\star(i-1))\|^2+ M(n), \quad n\geq 1,
\end{multline}
where
\begin{gather} \label{def.Yi}
Y^{(j)}(n)= 2\sqrt{h} \sum_{k=1}^d [X_h^\star(n)]_k  [\sigma_h(n)]_{kj}, \quad j=1,\ldots,r, \quad n\geq 1,\\
\label{def.M2}
M(n)=\sum_{i=1}^n \sum_{j=1}^r Y^{(j)}(i-1)\xi^{(j)}(i), \quad n\geq 1.
\end{gather}
\end{lemma}
\begin{proof}
Notice that with $Y^{(j)}$ as defined in \eqref{def.Yi} and $M$ as defined in \eqref{def.M2}, we have
\begin{align*}
M(n)
&= \sum_{i=1}^n \sum_{j=1}^r \left(2\sqrt{h} \sum_{k=1}^d [X_h^\star(i-1)]_k  [\sigma_h(i-1)]_{kj}\right)\xi^{(j)}(i)\\
&= \sum_{i=1}^n \sum_{k=1}^d 2\sqrt{h} [X_h^\star(i-1)]_k \sum_{j=1}^r [\sigma_h(i-1)]_{kj}\xi^{(j)}(i)\\
&= \sum_{i=1}^n \sum_{k=1}^d 2\sqrt{h} [X_h^\star(i-1)]_k [\sigma_h(i-1)\xi(i)]_k,
\end{align*}
so that $M$ defined by \eqref{def.M2} obeys
\begin{equation} \label{eq.M2innerprod}
M(n)=2\sqrt{h} \sum_{i=1}^n \langle X_h^\star(i-1), \sigma_h(i-1)\xi(i)\rangle, \quad n\geq 1.
\end{equation}
Next, we rewrite \eqref{eq.SSXast} according to
$X_h(n)=X_h^\star(n)+hf(X_h^\star(n))$. Then
\begin{equation} \label{eq.Xh2Xhst2}
\|X_h(n)\|^2=\|X_h^\star(n)\|^2+2h\langle f(X_h^\star(n)), X_h^\star(n)\rangle + h^2\|f(X_h^\star(n))\|^2.
\end{equation}
From \eqref{eq.SSupdate}, for $n\geq 0$ we get
\[
\|X_h(n+1)\|^2=\|X_h^\star(n)\|^2+h\|\sigma_h(n)\xi(n+1)\|^2+ 2\sqrt{h}\langle X_h^\star(n), \sigma_h(n)\xi(n+1)\rangle,
\]
so by using \eqref{eq.Xh2Xhst2} we get
\begin{multline} \label{eq.Xnp12Xn2}
\|X_h(n+1)\|^2=\|X_h(n)\|^2 - 2h\langle f(X_h^\star(n)), X_h^\star(n)\rangle - h^2\|f(X_h^\star(n))\|^2
\\+h\|\sigma_h(n)\xi(n+1)\|^2+ 2\sqrt{h}\langle X_h^\star(n), \sigma_h(n)\xi(n+1)\rangle.
\end{multline}
Therefore for $n\geq 1$, by summing on both sides, and using \eqref{eq.M2innerprod} we have
\begin{multline*}
\|X_h(n)\|^2=\|X_h(0)\|^2 +\sum_{i=1}^n h\left\{-2\langle f(X_h^\star(i-1)), X_h^\star(i-1)\rangle + \|\sigma_h(i-1)\xi(i)\|^2\right\}
\\- \sum_{i=1}^n h^2\|f(X_h^\star(i-1))\|^2+ M(n),
\end{multline*}
where $M$ is defined in \eqref{def.M2}, as claimed.
\end{proof}
\subsection{A result on the asymptotic behaviour of martingales}
We prove now a useful lemma on the asymptotic behaviour of a martingale built from $\xi$ and sequences adapted to its natural filtration.
It is based on a result of Bramson, Questel and Rosenthal~\cite[Theorem 1.1]{bramsquestrosen:2004}.
\begin{lemma}  \label{lemma.bramsquesrosen}
Let $M=\{M(n):n\geq 1\}$ be a martingale with respect to the filtration $(\mathcal{F}(n))_{n\geq 0}$ of $\sigma$--fields on a probability space
$(\Omega,\mathcal{F},\mathbb{P})$ such that
\[
M(n)=\sum_{i=1}^n Y(i), \quad n\geq 1.
\]
If there exists a constant $K\in [1,\infty)$  such that
\begin{equation} \label{eq.MnsqMnabsineq}
\mathbb{E}[Y(n)^2|\mathcal{F}(n-1)] \leq K \mathbb{E}[|Y(n)||\mathcal{F}(n-1)]^2, \quad \text{a.s. for all $n\geq 1$},
\end{equation}
then
\begin{multline}  \label{eq.Mlimitoroscillate}
\{\omega: \lim_{n\to\infty} M(n,\omega) \text{ exists and is finite} \}
\\\cup \{\omega: \liminf_{n\to\infty} M(n,\omega)=-\infty,\quad \limsup_{n\to\infty} M(n,\omega)=+\infty\} \text{ is an a.s. event}
\end{multline}
\end{lemma}
We now prove a consequence of Lemma~\ref{lemma.bramsquesrosen}.
\begin{lemma} \label{lemma.recurrentmartingale}
Suppose that $\xi$ obeys Assumption~\ref{ass.normal}.
Suppose that $Y^{(j)}=\{Y^{(j)}(n):n\geq 0\}$ for $j=1,\ldots,r$ are sequences of $\mathcal{F}^\xi(n)$--measurable random variables.
Define $M=\{M(n):n\geq 1\}$
\begin{equation} \label{eq.M}
M(n)=\sum_{i=1}^n \sum_{j=1}^r Y^{(j)}(i-1)\xi^{(j)}(i), \quad n\geq 1.
\end{equation}
Then $M$ obeys \eqref{eq.Mlimitoroscillate}.
\end{lemma}
\begin{proof}[Proof of Lemma~\ref{lemma.recurrentmartingale}]
Define
\[
Y(n)=\sum_{j=1}^r Y^{(j)}(n-1)\xi^{(j)}(n), \quad n\geq 1.
\]
Since $Y^{(j)}(n-1)$ is $\mathcal{F}^\xi(n-1)$ measurable, and $\xi$ obeys Assumption~\ref{ass.normal}, it follows that
\[
\mathbb{E}\left[Y(n)^2|\mathcal{F}^\xi(n-1)\right]= \sum_{j=1}^r Y^{(j)}(n-1)^2=:\varsigma^2(n).
\]
Next, we recall that if $Z$ is a normal random variable with mean zero and variance $c^2$, then
\[
\mathbb{E}[|Z|]^2=\frac{1}{2\pi} c^2.
\]
Since $\xi^{(j)}(n)$ for $j=1,\ldots,r$ are independent standard normal random variables, and $Y^{(j)}(n-1)$ is $\mathcal{F}^\xi(n-1)$ measurable,
it follows that, conditional on $\mathcal{F}^\xi(n-1)$, $Y(n)$ is normally distributed with zero mean and variance $\varsigma^2(n)$.
Therefore
\[
\mathbb{E}\left[|Y(n)||\mathcal{F}^\xi(n-1)\right]^2=\frac{1}{2\pi} \varsigma^2(n)
=\frac{1}{2\pi} \mathbb{E}\left[Y(n)^2|\mathcal{F}^\xi(n-1)\right],
\]
so \eqref{eq.MnsqMnabsineq} holds with $K=2\pi$. Therefore all the hypotheses of Lemma~\ref{lemma.bramsquesrosen} apply to $M$, and so we have
the claimed conclusion \eqref{eq.Mlimitoroscillate}.
\end{proof}

We employ one other result from the convergence theory of discrete process. It appears as Lemma 2
in \cite{JAARXM:2006a}.
\begin{lemma}\label{lemma.nonegdif}
 Let $\{Z(n)\}_{n\in \mathbb{N}}$ be a non-negative $\mathcal{F}(n)$-measurable
 process, $\mathbb{E}|Z(n)|<\infty$ for all $n\in \mathbb{N}$ and
\begin{equation}
 Z(n+1)\leq Z(n)+W(n)-V(n)+\nu(n+1), \quad n = 0, 1, 2, \dots,
\label{reprsemm}
\end{equation}
where $\{\nu(n)\}_{n\in \mathbb{N}}$ is an $\mathcal{F}(n)$-martingale--difference,
$\{W(n)\}_{n\in \mathbb{N}}$, $\{V(n)\}_{n\in \mathbb{N}}$ are  nonnegative
$\mathcal{F}(n)$--measurable processes, $\mathbb{E}|W(n)|<+\infty$, $\mathbb{E}|V(n)|<+\infty$ for all
$n\in \mathbb{N}$.   Then
$$\left\{\omega: \sum_{n=1}^{\infty} W(n)<+\infty\right\}\subseteq
\left\{\omega: \sum_{n=1}^{\infty} V(n)<+\infty\right\}\bigcap\{Z(n)\to\},$$
where $\{Z(n)\to\}$ denotes the set of all $\omega\in\Omega$
for which $\lim\limits_{n\to \infty} Z(n,\omega)$ exists and is
finite.
\end{lemma}

\subsection{Proof of Lemma~\ref{lemma.Uasy}} For $n\geq 1$, we have
$[U_h(n)]_i=\sqrt{h}\sum_{j=1}^r [\sigma_h(n-1)]_{ij}\xi_j(n+1)$.
Hence $[U_h(n)]_i$ is normally distributed with mean zero and
variance $\theta_i(n)^2:=h\sum_{j=1}^r [\sigma_h(n-1)]_{ij}^2$.
Therefore,
\begin{equation} \label{eq.Uniprob}
\mathbb{P}[|[U_h(n)]_i|\geq \epsilon]=1 -
\Phi\left(\frac{\epsilon}{\theta_i(n)}\right).
\end{equation}
Define $\theta^2(n)=\sum_{i=1}^d
\theta_i(n)^2=h\|\sigma_h(n-1)\|^2_F$. Since $\theta^2(n)\geq
\theta_i(n)^2$ for each $i=1,\ldots,d$, we have
\begin{equation*}
\sum_{i=1}^d
\left\{1-\Phi\left(\frac{\epsilon}{\theta_i(n)}\right)\right\}\leq d
\left(1-\Phi\left(\frac{\epsilon}{\theta(n)}\right)\right).
\end{equation*}
Suppose, for each $n$, that $Z_i(n)$ for $i=1,\ldots,d$ are
independent standard normal random variables. Define
$Z(n)=(Z_1(n),Z_2(n),\ldots,Z_d(n))$ and suppose that $(Z(n))_{n\geq
0}$ are a sequence of independent normal vectors. Define finally
\[
X_i(n)=\theta_i(n)Z_i(n), \quad X(n)=\sum_{i=1}^d X_i(n), \quad
n\geq 0.
\]
Then we have that $X_i$ is a zero mean normal with variance
$\theta_i^2$ and $X$ is a zero mean normal with variance $\theta^2$.
Define $Z^\ast(n)=X(n)/\theta(n)$ is a standard normal random
variable. Therefore we have that
\begin{equation} \label{eq.thth10}
\mathbb{P}[|X(n)|>\epsilon]=\mathbb{P}[|Z^\ast(n)|\geq
\epsilon/\theta(n)]=2\mathbb{P}[Z^\ast(n)\geq \epsilon/\theta(n)]
=2\left(1-\Phi\left(\frac{\epsilon}{\theta(n)}\right)\right).
\end{equation}
With $A_i(n)=\{ |X_i(n)|\leq \epsilon/d\}$, $B(n)=\{\sum_{i=1}^d
|X_i(n)|\leq \epsilon\}$, then $\cap_{i=1}^d A_i(n) \subseteq B(n)$,
so
\[
\mathbb{P}\left[|X(n)|>\epsilon\right]\leq
\mathbb{P}[\overline{B}(n)]\leq
\mathbb{P}\left[\overline{\cap_{i=1}^d A_i(n)}\right]
=\mathbb{P}\left[\cup_{i=1}^d \overline{A_i(n)}\right] \leq
\sum_{i=1}^d\mathbb{P}\left[\overline{A_i(n)}\right].
\]
Since $X_i=\theta_i Z_i$, we have \begin{equation} \label{eq.thth13}
\mathbb{P}\left[|X(n)|>\epsilon\right]\leq
2\sum_{i=1}^d\mathbb{P}\left[X_i(n)\geq \epsilon/d\right]
=2\sum_{i=1}^d \left\{ 1-\Phi\left(\frac{\epsilon/d}{\theta_i(n)}
\right) \right\}.
\end{equation}
By \eqref{eq.thth10} and \eqref{eq.thth13}, we get 
\begin{equation} \label{eq.phi2}
1-\Phi\left(\frac{\epsilon}{\theta(n)}\right) \leq \sum_{i=1}^d
\left\{ 1-\Phi\left(\frac{\epsilon/d}{\theta_i(n)} \right) \right\}.
\end{equation}
Define $\|U_h(n)\|_1=\sum_{i=1}^d |[U_h(n)]_i|$ for $n\geq 1$.
Therefore, as $\|U_h(n)\|_1\geq |[U_h(n)]_i|$, we have that
$\mathbb{P}[\|U_h(n)\|_1\geq \epsilon]\geq \mathbb{P}[|[U_h(n)]_i|\geq
\epsilon]$ for each $i=1,\ldots,d$. Therefore by \eqref{eq.Uniprob}
and \eqref{eq.phi2}, we have
\begin{equation} \label{eq.normUvsUi1}
d\mathbb{P}[\|U_h(n)\|_1\geq \epsilon]\geq \sum_{i=1}^d
\mathbb{P}[|[U_h(n)]_i|\geq \epsilon] =\sum_{i=1}^d \left\{1 -
\Phi\left(\frac{\epsilon}{\theta_i(n)}\right)\right\} \geq
1-\Phi\left(\frac{d\epsilon}{\theta(n)}\right).
\end{equation} On the other hand, defining
$A_i(j)=\{|[U_h(n)]_i|\leq \epsilon/d\}$ and $B(j)=\{\|U_h(n)\|_1\leq
\epsilon\}$, we see that $\cap_{i=1}^d A_i(n)\subseteq B(n)$. Then
\begin{multline*}
\mathbb{P}[\|U_h(n)\|_1\geq \epsilon]=\mathbb{P}[\overline{B(n)}]\leq
\mathbb{P}\left[\overline{\cap_{i=1}^d A_i(n)}\right]\\
=\mathbb{P}\left[\cup_{i=1}^d \overline{A_i(n)}\right] \leq
\sum_{i=1}^d \mathbb{P}\left[|[U_h(n)]_i|\geq \epsilon/d\right].
\end{multline*}
Hence by \eqref{eq.Uniprob} and \eqref{eq.phi2} we get
\begin{multline}   \label{eq.normUvsUi2}
\mathbb{P}[\|U_h(n)\|_1\geq \epsilon]\leq \sum_{i=1}^d
\mathbb{P}\left[|[U_h(n)]_i|\geq \epsilon/d\right] = \sum_{i=1}^d
\left\{1 - \Phi\left(\frac{\epsilon/d}{\theta_i(n)}\right)\right\}\\
\leq d \left(1-\Phi\left(\frac{\epsilon/d}{\theta(n)}\right)\right).
\end{multline}

Part (A). Suppose $S_h(\epsilon)<+\infty$ for all $\epsilon>0$. Then,
by \eqref{eq.normUvsUi2} we have that
\[
\mathbb{P}[\|U_h(n)\|_1\geq \epsilon]<+\infty
\]
and so by the Borel--Cantelli lemma, $\limsup_{n\to\infty}
\|U_h(n)\|_1 \leq \epsilon$ a.s. for each $\epsilon>0$. Letting
$\epsilon\downarrow 0$ through the rational numbers gives
$\lim_{n\to\infty} U_h(n)=0$ a.s.

Part (B). Suppose $S_h(\epsilon)=+\infty$ for all $\epsilon>0$. Then,
by \eqref{eq.normUvsUi1} we have that
\[
\mathbb{P}[\|U_h(n)\|_1\geq \epsilon]=+\infty
\]
Since $(\|U_h(n)\|_1)_{n\geq 1}$ is a sequence of independent random
variables, by the Borel--Cantelli lemma we have that
$\limsup_{n\to\infty} \|U_h(n)\|_1 \geq \epsilon$ a.s. for each
$\epsilon>0$. Letting $\epsilon\to \infty$ through the integers
gives $\limsup_{n\to\infty} \|U_h(n)\|=+\infty$ a.s.

Part (C). Suppose $S_h(\epsilon)<+\infty$ for all
$\epsilon>\epsilon'$. If $\epsilon>\epsilon'$, then by
\eqref{eq.normUvsUi2} we have
\begin{equation*}
\sum_{n=1}^\infty \mathbb{P}[\|U_h(n)\|_1\geq d h\epsilon]\leq
\sum_{n=0}^\infty d
\left(1-\Phi\left(\frac{\epsilon}{\|\sigma_h(n)\|_F}\right)\right)<+\infty,
\end{equation*}
and so $\limsup_{n\to\infty} \|U_h(n)\|_1\leq d h\epsilon'=:c_2$,
a.s. On the other hand, if $\epsilon<\epsilon'$, by
\eqref{eq.normUvsUi1} we get
\begin{equation*}
\sum_{n=1}^\infty \mathbb{P}[\|U_h(n)\|_1\geq h\epsilon/d]\geq
\sum_{n=0}^\infty \frac{1}{d} \left\{1-\Phi\left(\frac{\epsilon}{
\|\sigma_h(n)\|_F}\right)\right\}=+\infty.
\end{equation*}
Therefore, using the Borel--Cantelli lemma and independence of
$\|U_h(n)\|_1$, we have that $\limsup_{n\to\infty} \|U_h(n)\|_1\geq
h\epsilon'/d=:c_2$, a.s.

\section{Proof of  Theorem~\ref{thm.sig2l1Xto0}}
Recall from Lemma~\ref{lemma.repX2} that $X_h$ obeys \eqref{eq.X2sum}
with $M$ given by \eqref{eq.M2innerprod}.
Since $f$ obeys \eqref{eq.fglobalunperturbed}, this implies that
\[
\|X_h(n)^2\|-\|X_h(0)\|^2\leq \sum_{i=1}^{n}h\|\sigma_h(i-1)\xi(i)\|^2+M(n), \quad n\geq 1.
\]
We want to prove that $\limsup_{n\to \infty}\|X_h(n)\|<+\infty$, therefore we need to prove that
$\limsup_{n\to \infty}\sum_{i=1}^{n}h\|\sigma_h(i-1)\xi(i)\|^2<+\infty$ and $\limsup_{n\to \infty}M(n)<+\infty$.
Define $P(n)=\sum_{i=1}^{n}h\|\sigma_h(i-1)\xi(i)\|^2$. Since $(P(n))_{n\geq 1}$ is
a non--decreasing sequence, we have that $P_{\infty}=\lim_{n\to\infty} P(n)$ exists a.s.
We wish to show that $P_\infty$ must be finite a.s. Suppose to the contrary that there is an
event $A=\{\omega: P_{\infty}(\omega)=\infty\}$ with $\mathbb{P}[A]>0$.
Then as $P_\infty$ is a non--negative random variable, we have that
$\mathbb{E}[P_{\infty}]=+\infty$. However by Fubini's Theorem we have
\begin{align*}
\mathbb{E}[P_{\infty}]&=\mathbb{E}\sum_{i=1}^{\infty}\|\sigma_h(i-1)\xi(i)\|^2=
\sum_{i=1}^{\infty}\|\sigma_h(i-1)\|^2_F<+\infty,
\end{align*}
which is a contradiction.
Therefore it must be that $\lim_{n\to \infty}P(n)= P_{\infty}$ 
exists and is finite a.s.
From \eqref{eq.Xnp12Xn2} and \eqref{eq.fglobalunperturbed} we have
\begin{multline} \label{eq.delXn2}
\|X_h(n+1)\|^2-\|X_h(n)\|^2 
\leq  h\|\sigma_h(n)\xi(n+1)\|^2+2\sqrt{h}\langle X_h^\star(n),\sigma_h(n)\xi(n+1)\rangle.
\end{multline}
We know that $\mathbb{E}[\|X_h(0)\|^2]<+\infty$. We wish to prove that $\mathbb{E}[\|X_h(n)\|^2]<+\infty$ for
each $n\in\mathbb{N}$, which we prove by induction. Suppose that $\mathbb{E}[\|X_h(n)\|^2]<+\infty$.
Then, we get
\begin{multline*}
\mathbb{E}[\|X_h(n+1)\|^2]\leq
\mathbb{E}[\|X_h(n)\|^2]+\mathbb{E}[h\|\sigma_h(n)\xi(n+1)\|^2]
\\+2\sqrt{h}\mathbb{E}[\langle X_h^\star(n),\sigma_h(n)\xi(n+1)\rangle].
\end{multline*}
We now compute the second term on the right--hand side. Because $X_h^\star(n)$ depends on $X_h(n)$ and is $\mathcal{F}(n)$--measurable, and $\xi(n+1)$ is $\mathcal{F}(n+1)$--measurable and independent of $\mathcal{F}(n)$, therefore $\xi(n+1)$ is independent of $X_h^\star(n)$. Moreover $\mathbb{E}[\|X_h^\star(n)\|]\leq \mathbb{E}[\|X_h(n)\|]<\infty$ and similarly $\mathbb{E}[\|\xi(n+1)\|^2]$ is finite.
We get
\begin{align*}
\mathbb{E}[\langle X_h^\star(n),\sigma_h(n)\xi(n+1)\rangle]
&= \mathbb{E}\left[ \sum_{i=1}^d  [X_h^\star(n)]_i [\sigma_h(n)\xi(n+1)]_i\right]\\
&= \mathbb{E}\left[ \sum_{i=1}^d  [X_h^\star(n)]_i  \sum_{j=1}^r [\sigma_h(n)]_{ij}\xi_j(n+1) \right]\\
&= \mathbb{E}\left[ \sum_{j=1}^r \left(\sum_{i=1}^d  [X_h^\star(n)]_i [\sigma_h(n)]_{ij}\right)\xi_j(n+1) \right]\\
&=  \sum_{j=1}^r \mathbb{E}\left[\left(\sum_{i=1}^d  [X_h^\star(n)]_i [\sigma_h(n)]_{ij}\right)\xi_j(n+1) \right].
\end{align*}
Since $\mathbb{E}[\|X_h^\star(n)\|^2]<+\infty$ and $\mathbb{E}[\|\xi(n+1)\|^2]<+\infty$ and $\sigma_h$ is deterministic,
it follows from independence and the fact that $\mathbb{E}[\xi_j(n+1)]=0$ for all $n$ and $j$, that
\[
\mathbb{E}\left[\left(\sum_{i=1}^d  [X_h^\star(n)]_i [\sigma_h(n)]_{ij}\right)\xi_j(n+1) \right]
= \mathbb{E}\left[\sum_{i=1}^d  [X_h^\star(n)]_i [\sigma_h(n)]_{ij}\right]\mathbb{E}[\xi_j(n+1)]=0.
\]
Hence
\begin{equation*}
\mathbb{E}[\langle X_h^\star(n),\sigma_h(n)\xi(n+1)\rangle]=0.
\end{equation*}
Next, we return to $P(n)$ to get
\begin{multline*}
\mathbb{E}[\|\sigma_h(n)\xi(n+1)\|^2]
=\mathbb{E}\sum_{i=1}^d [\sigma_h(n)\xi_i(n+1)]_i^2
=\mathbb{E}\sum_{i=1}^d \left(\sum_{j=1}^r [\sigma_h(n)]_{ij}\xi_j(n+1)\right)^2\\
=\mathbb{E}\sum_{i=1}^d \left\{ \sum_{j=1}^r [\sigma_h(n)]_{ij}^2\xi_j^2(n+1)
+\sum_{j}\sum_{k\neq j} [\sigma_h(n)]_{ij} \sigma_{ik}(n) \xi_j(n+1)\xi_k(n+1)\right\}.
\end{multline*}
By the independence of $\xi_j(n+1)$, $\xi_i(n+1)$ for $i\neq j$, we have
\begin{equation} \label{eq.varsigxi}
\mathbb{E}[\|\sigma_h(n)\xi(n+1)\|^2]
=\sum_{i=1}^d \sum_{j=1}^r [\sigma_h(n)]_{ij}^2= \|\sigma_h(n)\|_F^2.
\end{equation}
Therefore
\[
\mathbb{E}[\|X_h(n+1)\|^2]\leq \mathbb{E}[\|X_h(n)\|^2]+h\|\sigma(n)\|_F^2<+\infty.
\]
Thus by induction we have $\mathbb{E}[\|X_h(n+1)\|^2]<+\infty$ for all $n\in\mathbb{N}$.
Now by \eqref{eq.delXn2} we get
\begin{align*}
\|X_h(n)\|^2-\|X_h(0)\|^2&=\sum_{j=0}^{n-1} \{\|X_h(j+1)\|^2-\|X_h(j)\|^2\}\\
&\leq h\sum_{j=0}^{n-1} \|\sigma_h(j)\xi(j+1)\|^2+2\sqrt{h}\sum_{j=0}^{n-1}\langle X_h^\star(j),\sigma_h(j)\xi(j+1)\rangle\\
&=hP(n)+2\sqrt{h}\sum_{j=0}^{n-1}\langle X_h^\star(j),\sigma_h(j)\xi(j+1)\rangle.
\end{align*}
Because $\mathbb{E}[\|X_h(n)\|^2]<+\infty$ and $\mathbb{E}[\|X_h^\star(n)\|^2]\leq\mathbb{E}[\|X_h(n)\|^2]$, thus $\mathbb{E}[\|X_h^\star(n)\|^2]<+\infty$ for all. Therefore
\[
M(n)=\sum_{j=0}^{n-1}2\sqrt{h}\langle X_h^\star(j),\sigma_h(j)\xi(j+1)\rangle
\]
is a martingale. Next we compute the quadratic variation of $M$. To this end, we may write $M$ according to
\begin{align*}
M(n) 
&=2\sqrt{h}\sum_{j=0}^{n-1}\sum_{l=1}^rQ_l(j)\xi_l(j+1),
\end{align*}
where $Q_l(j)=\sum_{i=1}^d [X_h^\star(j)]_i [\sigma_h(j)]_{il}$. Thus
$M(j+1)-M(j)=$$2\sqrt{h}\sum_{l=1}^rQ_l(j)\xi_l(j+1)$. Hence the quadratic variation of $M$ is given by
\begin{align*}
\langle M\rangle(n) 
&=4h\sum_{j=0}^{n-1}\mathbb{E}\left[\left(\sum_{l=1}^rQ_l(j)\xi_l(j+1)\right)^2\Bigg|\mathcal{F}_j\right]\\
&=4h\sum_{j=0}^{n-1}\mathbb{E}[\sum_{l=1}^rQ_l(j)^2\xi_l(j+1)^2\\
&\qquad+\sum_{m=1}^r \sum_{l\neq m}Q_l(j)Q_m(j)\xi_l(j+1)\xi_m(j+1)|\mathcal{F}_j]\\
&=4h\sum_{l=1}^rQ_l(j)^2\mathbb{E}[\xi_l(j+1)^2|\mathcal{F}_j]\\
&\qquad+\sum_{m=1}^r \sum_{l\neq m}Q_l(j)Q_m(j)\mathbb{E}[\xi_l(j+1)\xi_m(j+1)|\mathcal{F}_j]\\
&=4h\sum_{j=0}^{n-1}\sum_{l=1}^rQ_l^2(j).
\end{align*}
Therefore, by using the Cauchy--Schwartz inequality, we obtain the estimate
\begin{multline} \label{eq.estsqvM}
\langle M\rangle(n)
\leq
4h\sum_{j=0}^{n-1}\sum_{l=1}^r\left\{\sum_{i=1}^d X_h^\star(j)^2_i\sum_{i=1}^d\sigma_{il}^2(j)\right\}\\
=4h\sum_{j=0}^{n-1}\left(\sum_{i=1}^dX_i^\star(j)^2\right)\cdot \sum_{l=1}^r\sum_{i=1}^d\sigma_{il}^2(j)
=4h\sum_{j=0}^{n-1}\|X_h^\star(j)\|^2\|\sigma_h(j)\|^2_F.
\end{multline}
Define the events
\begin{gather*}
A_1=\{\omega: \lim_{n\to\infty} P(n,\omega) = P_{\infty}\in (0,\infty)\}, \quad
A_2=\{\omega:\lim_{n\to \infty}\langle M\rangle(n)=+\infty\}.
\end{gather*}
Suppose that $P[A_2]>0$. Let $A_3=A_1\cap A_2$,
so that $\mathbb{P}[A_3]>0$. Then a.s. on $A_3$ we have
\[
\lim_{n\to \infty}\frac{M(n)}{\langle M\rangle(n)}=0.
\]
Next suppose that $\epsilon \in (0,1)$ is so small that
\begin{equation}\label{eq.epsilonsmall}
4\epsilon h \sum_{n=1}^{\infty}\|\sigma_h(n)\|^2_F<\frac{1}{2}.
\end{equation}
Thus for every $\omega\in A_3$ and for every $\epsilon<1$, there is an $N(\omega,\epsilon)>1$ such that $|M(n,\omega)|\leq \epsilon \langle M \rangle(n,\omega)$ for all $n\geq N(\omega,\epsilon)$.
Therefore for $n\geq N(\omega,\epsilon)$ we have
\begin{align*}
\|X_h(n,\omega)\|^2&\leq \|X_h(0,\omega)\|^2+hP(n,\omega)+M(n,\omega)\\
&\leq \|X_h(0,\omega)\|^2+hP_{\infty}(\omega)+\epsilon\langle M\rangle(n,\omega).
\end{align*}
Since $\|X_h(n,\omega)\|^2\leq \max_{0\leq j\leq N(\omega,\epsilon)}\|X_h(j,\omega)\|^2=:X^{\star \star}_h(\epsilon,\omega)<+\infty$ for $0\leq n\leq N(\omega,\epsilon)$. Define $C_1(\epsilon,\omega):=\|X_h(0,\omega)\|^2+hP_{\infty}(\omega)+X_h^{\star \star}(\epsilon,\omega)$ which is finite. Therefore
\[
\|X_h(n,\omega)\|^2\leq C_1(\epsilon,\omega)+\epsilon \langle M\rangle(n,\omega), \quad n\geq 1.
\]
We drop the $\omega$--dependence temporarily. Define
$y(n)=\|\sigma_h(n)\|^2_F\|X_h(n)\|^2$ for $n\geq 0$. Hence by the last inequality and
\eqref{eq.estsqvM}, we have
\[
y(n)=\|\sigma_h(n)\|^2_F\|X_h(n)\|^2\leq C_1(\epsilon)\|\sigma_h(n)\|^2_F+4\epsilon h\|\sigma_h(n)\|^2_F\sum_{j=0}^{n-1}y(j), \quad n\geq 1,
\]
where we have used the fact that $\|X_h^\star(j)\|\leq \|X_h(j)\|^2$ for all $j\geq 0$.
Thus for $m\geq 1$ we have
\begin{align*}
\sum_{n=1}^my(n)&\leq C_1(\epsilon)\sum_{n=1}^m\|\sigma_h(n)\|^2_F
+4\epsilon h\sum_{n=1}^m\|\sigma_h(n)\|^2_F\sum_{j=0}^{n-1}y(j)\\
&\leq C_1(\epsilon)\sum_{n=1}^m\|\sigma_h(n)\|^2_F+4\epsilon h\sum_{n=1}^m\|\sigma_h(n)\|^2_F\sum_{j=0}^my(j)\\
&\leq C_1(\epsilon)\sum_{n=1}^{\infty}\|\sigma_h(n)\|^2_F+4\epsilon h
\sum_{n=1}^{\infty}\|\sigma_h(n)\|^2_F\left(\sum_{j=1}^my(j)+y(0)\right)\\
&\leq C(\epsilon)\sum_{n=1}^{\infty}\|\sigma_h(n)\|^2_F+\frac{1}{2}\sum_{j=1}^my(j)
\end{align*}
where $C(\epsilon)=C_1(\epsilon)+4\epsilon h\|X_h(0)\|^2\|\sigma_h(0)\|^2$, condition \eqref{eq.epsilonsmall} was used at the last step, and the non-negativity and definition of $y$ was used. Hence $\sum_{j=1}^my(j)\leq 2C(\epsilon)\sum_{n=1}^{\infty}\|\sigma_h(n)\|^2$ for all $m\geq 1$. Thus
\[
\sum_{n=1}^{\infty}\|\sigma_h(n)\|^2\|X_h(n,\omega)\|^2<+\infty \quad\text{for each $\omega \in A_3$}.
\]
This implies $\lim_{n\to \infty}\langle M\rangle(n,\omega)<+\infty$ for each $\omega\in A_3$, which is a contradiction.
Therefore we have that $\mathbb{P}[A_2]=0$. Thus we have that
\[
\lim_{n\to \infty}\langle M\rangle(n) \quad\text{exists and is a.s. finite}
\]
This implies $\lim_{n\to \infty}M(n)$ exists and is finite a.s., and so $\limsup_{n\to \infty}\|X_h(n)\|<+\infty$ a.s.

Next we show that $\lim_{n\to \infty}\|X_h(n)\|^2=:L\in [0,+\infty)$
a.s. To do this we apply Lemma~\ref{lemma.nonegdif} with
%
$Z(n+1):=\|X_h(n+1)\|^2$, $Z(n):=\|X_h(n)\|^2$, $V(n):=0$, $W(n):=h\|\sigma_h(n)\xi(n+1)\|^2$, $\nu(n+1):=2\sqrt{h}\langle X_h^\star(n),\sigma_h(n) \xi(n+1)\rangle$.
Therefore, by \eqref{eq.varsigxi} we get
 $\mathbb{E}[\sum_{n=1}^{\infty}W(n)]=\sum_{n=1}^{\infty}h\|\sigma_h(n)\|^2_F<+\infty$, which implies that
$\sum_{n=1}^{\infty}W(n)<+\infty$ a.s. Therefore,
$\lim_{n\to \infty}\|X_h(n)\|^2=:L\in [0,\infty)$ a.s. Moreover, as $W(n)\geq 0$ it also follows that $\lim_{n\to \infty}W(n)=0$ a.s.,  so $\lim_{n\to \infty} U_h(n)=0$ a.s.

We are now in a position to prove that $X_h(n)\to 0$ as $n\to\infty$ a.s.
Recall from \eqref{eq.Xhnp12Xhn2} and that $U_h(n)\to 0$ as $n\to \infty$. Since $\|X_h(n)\|\to \sqrt{L}$ as $n\to \infty$,
it follows that $\|X^\ast_h(n)\|=\|X_h(n+1)-U_h(n+1)\|\to \sqrt{L}$ as $n\to \infty $.
Hence $|\langle X^\ast_h(n),U_h(n+1)\rangle|\leq\|X^\ast_h(n)\|\|U_h(n+1)\|\to 0$ as $n\to \infty$.
Therefore, rearranging \eqref{eq.Xhnp12Xhn2} gives
\begin{multline*}
2h\langle f(X^\ast_h(n)),X^\ast_h(n)\rangle+h^2\|f(X_h^\ast(n)\|^2\\
=\|X_h(n)\|^2-\|X_h(n+1)\|^2+\|U_h(n+1)\|^2+2\langle X_h^\ast(n),U_h(n+1)\rangle
\end{multline*}
which goes to 0 as $n\to \infty$.
Thus $\lim_{n\to \infty}\left\{2\langle f(X^\ast_h(n)),X^\ast_h(n)\rangle+h\|f(X_h^\ast(n)\|^2\right\}=0$.
Next define $R:\mathbb{R}^d\to \mathbb{R}$ by
\begin{equation}\label{def.R}
R(x)=2\langle x,f(x)\rangle +h\|f(x)\|^2, \quad x\in \mathbb{R}^d.
\end{equation}
Then we have $R(0)=0$, $x\mapsto R(x)$ is continuous, $R(x)>0$ for all $x\neq 0$.
Therefore we have $\lim_{n\to \infty}R(X_h^\ast(n))=0$ and $\lim_{n\to \infty}\|X_h(n)\|=\sqrt{L}$.
Thus
\[
R(X^\ast_h(n))\geq \inf_{\|x\|=\|X^\ast_h(n)\|}R(x)\geq 0.
\]
Hence $0=\limsup_{n\to\infty}R(X^\ast_h(n))\geq \limsup_{n\to \infty}\inf_{\|x\|=\|X_h^\ast(n)\|}R(x)\geq 0$. Therefore
\[
\lim_{n\to\infty}\inf_{\|x\|=\|X^\ast_h(n)\|}R(x)=0.
\]
Now define $R^\ast:\mathbb{R}^+\to \mathbb{R}$ by $R^\ast(y)=\inf_{\|x\|=y}R(x)$. Since $R$ is continuous, $R^\ast$ is continuous. Thus,
because $\lim_{n\to \infty}R^\ast(\|X_h^\ast(n)\|)=0$ and $\|X^\ast_h(n)\|\to \sqrt{L}$ as $n\to \infty$, we have that
\[
0=\lim_{n\to\infty}R^\ast(\|X^\ast_h(n)\|)=R^\ast \left(\lim_{n\to\infty}\|X^\ast_h(n)\|\right)=R^\ast(\sqrt{L}).
\]
Thus $\inf_{\|x\|=\sqrt{L}}R(x)=0$. Since $R$ is continuous, there exists $X^\ast$ with $\|X^\ast\|=\sqrt{L}$ such that $R(x^\ast)=0$,
but since $R(0)=0$ and $R(x)>0$ for all $x\neq 0$, this forces $x^\ast=0$, so $L=0$. Hence,
$\lim_{n\to \infty}\|X_h(n)\|^2=0$, a.s., as required.

\section{Proof of Theorems~\ref{theorem.XunboundedXboundedbelow}}
We start by proving part (A). Suppose that $A:=\{\omega: \limsup_{n\to\infty}
\|X_h(n,\omega)\|<+\infty\}$ is an event with $\mathbb{P}[A]>0$.
Define for $\omega\in A$ the quantity $L(\omega)\in [0,\infty)$ such
that $L(\omega)=\limsup_{n\to\infty} \|X_h(n,\omega)\|$. By
Lemma~\ref{lem.Xastcontract}, we have $\|X_h^\star(n)\|\leq
\|X_h(n)\|$ for all $n\geq 0$. Therefore, for every $\omega\in A$,
we have $\limsup_{n\to\infty} \|X_h^\star(n,\omega)\|\leq
L(\omega)$. By \eqref{eq.SSupdate}, we have
$U_h(n+1,\omega)=X_h(n+1,\omega)-X_h^\star(n,\omega)$. Since
$S_h(\epsilon)=+\infty$ for every $\epsilon>0$, by
Lemma~\ref{lemma.Uasy} the process $U_h$ given by \eqref{def.U}
obeys $\limsup_{n\to\infty}\|U_h(n)\|=+\infty$ a.s. Suppose
$\Omega_4$ is the a.s. event such that
$\Omega_4=\{\omega:\limsup_{n\to\infty} \|U_h(n,\omega)\|=+\infty\}$.
Then $A_1=A\cap \Omega_4$ is an event with $\mathbb{P}[A_1]>0$.
Therefore for $\omega\in A_1$ we have
\begin{align*}
+\infty&=\limsup_{n\to\infty}\|U_h(n+1,\omega)\|=\limsup_{n\to\infty}\|X_h(n+1,\omega)-X_h^\star(n,\omega)\|\\
&\leq
\limsup_{n\to\infty}\|X_h(n+1,\omega)\|+\limsup_{n\to\infty}\|X_h^\star(n,\omega)\|\leq 2L(\omega),
\end{align*}
a contradiction. Therefore we have that $\mathbb{P}[A]=0$, which proves part (A).

%
For the proof of part (B), because $S_h(\epsilon)<+\infty$ for all $\epsilon>\epsilon'$ and $S_h(\epsilon)=+\infty$ for all $\epsilon<\epsilon'$, Lemma~\ref{lemma.Uasy} implies that the process $U_h$ defined by \eqref{def.U} obeys $0<c_1\leq \limsup_{n\to\infty}\|U_h(n)\|\leq c_2<+\infty$ a.s. for some deterministic $c_1$ and $c_2$. In fact
 \[
 U_h^\ast(\omega):=\limsup_{n\to\infty} \|U_h(n,\omega)\|\in [c_1,c_2].
 \]
 Therefore, we know that $\limsup_{n\to\infty} \|X_n(n,\omega)\|>0$ for all $\omega\in \Omega_1$ where $\Omega_1$ is
 an almost sure event.

Let $\omega\in \Omega_1$. We have that
\[
0<c'(\omega):=\limsup_{n\to\infty} \|X_h(n,\omega)\|.
\]
Clearly $c''(\omega):=\limsup_{n\to\infty} \|X_h^\star(n,\omega)\|\leq c'(\omega)$, where the latter inequality holds by Lemma~\ref{lem.Xastcontract}.
We have that $c''(\omega)>0$, because if $X_h^\star(n,\omega)\to 0$ as $n\to\infty$, and $f$ obeys \eqref{eq.fglobalunperturbed}, we have
\[
\lim_{n\to\infty} X_h(n,\omega)=\lim_{n\to\infty} X_h^\star(n,\omega)+f(X_h^\star(n,\omega))=0.
\]
%
By \eqref{eq.SSupdate}, since $c'(\omega)\geq c''(\omega)$, we get
\begin{align*}
U_h^\ast(\omega)&=\limsup_{n\to\infty}\|U_h(n+1,\omega)\|\leq \limsup_{n\to\infty}\|X_h(n+1,\omega)\|+\|X_h^\star(n,\omega)\|\\
&=c'(\omega)+c''(\omega) \leq 2c'(\omega).
\end{align*}
Therefore $c'(\omega)\geq U_h^\ast(\omega)/2\geq c_1/2$, as required.

\section{Proof of Theorems~\ref{thm:Xlim0}, \ref{thm:Xlim0faway0}, and \ref{theorem.xto0}}
\subsection{Properties of the data}
Before we turn to the proof of Theorem~\ref{thm:Xlim0} we first require some auxiliary results concerning the function $f$.
\begin{lemma}\label{lemma.Finverse}
Suppose that $f\in C(\mathbb{R}^d);\mathbb{R}^d)$. Suppose Assumption~\ref{ass.existl} holds.
If $K>0$ and $\|x\|>K>0$, then every solution $x^\star$ of \eqref{eq.splitimplicit} obeys $\|x^\star\|>F^{-1}_h(K)>0$, where
\begin{equation} \label{def.Fhmult}
F_h(x):=x+h\sup_{\|u\|\leq x}\|f(u)\|, \quad x\geq 0.
\end{equation}
\end{lemma}
\begin{proof}
Since $F_h:[0,\infty)\to[0,\infty)$ is increasing, $F_h^{-1}$ is increasing. Let $K>0$ and define  $M=F_h^{-1}(K)>0$. Since $\|x\|>K=F_h(M)$, and $x^\star$ obeys $x=x^\star +hf(x^\star)$, we get
\begin{align*}
K<\|x\|&=\|x^\star+hf(x^\star)\|\leq \|x^\star\|+h\|f(x^\star)\|\\
&\leq \|x^\star\|+h\sup_{\|u\|\leq \|x^\star\|}\|f(u)\|=F_h(\|x^\star\|).
\end{align*}
Thus $K<F_h(\|x^\star\|)$, therefore $F_h^{-1}(K)<\|x^\star\|$, as required.
\end{proof}

\begin{lemma} \label{lemma:barfphi}
Suppose that $f$ obeys \eqref{eq.fglobalunperturbed}. Define $\bar{f}:[0,\infty)\to \mathbb{R}$ by
\begin{equation} \label{def.barf}
\bar{f}(y):=\inf_{\|x\|=y}\langle x,f(x)\rangle,
\end{equation}
and $\varphi:[0,\infty)\to \mathbb{R}$ by
\[
\varphi(y)=\inf_{x\in[F_h^{-1}(\frac{3y}{4}),\frac{5y}{4}]}
\bar{f}(x).
\]
where $F_h$ is defined by \eqref{def.Fhmult}. Then $\bar{f}(x)>0$ for all $x>0$ and $\varphi(x)>0$
for all $x>0$.
\end{lemma}
\begin{proof}
Since $f$ is continuous, it follows that $\bar{f}$ is continuous.
Also, as $F_h$ is continuous and invertible, $F_h^{-1}$ exists and
is continuous, and therefore $\varphi$ is continuous also. Notice
that the continuity of $f$ and the dissipative condition in
\eqref{eq.fglobalunperturbed} implies that $\bar{f}(y)>0$ for all
$y>0$. We show also that $\varphi(y)>0$ for $y>0$. Suppose to the
contrary that $\varphi(y)=0$ for some $y>0$. Then, as $\bar{f}$ is
continuous, there exists $x\in
[F_h^{-1}(\frac{3y}{4}),\frac{5y}{4}]$ such that $\bar{f}(x)=0$.
However, as $y>0$, we have that $F_h^{-1}(3y/4)>0$, and so this
implies that there is $x>0$ for which $\bar{f}(x)=0$.
\end{proof}
\subsection{Asymptotic results}
We are now ready to prove the first step of the main result of this section,
which is namely to establish that $\liminf_{n\to \infty}\|X_h(n)\|=0$.
\begin{lemma} \label{lemma:Xliminf0}
Suppose that $(X_h,X^\star_h)$ is a solution of \eqref{eq.splitstep}.
Suppose that $f$ obeys \eqref{eq.fglobalunperturbed}, 
and that the sequence $\xi$ obeys Assumption~\ref{ass.normal}. If $S_h(\epsilon)$ defined by \eqref{def.Seps}
obeys $S_h(\epsilon)<+\infty$ for all $\epsilon>0$, then
\[
\{\liminf_{n\to \infty}\|X_h(n)\|=0\}\cup \{\lim_{n\to\infty} \|X_h(n)\|=+\infty\} \text{ is an a.s. event}.
\]
\end{lemma}
\begin{proof}
Using \eqref{eq.Xnp12Xn2} together with \eqref{def.U} we get
\begin{multline}\label{eq.Xhnp12Xhn2}
\|X_h(n+1)\|^2=\|X_h(n)\|^2-2h\langle X_h^\star(n),f(X_h^\star(n))\rangle-h^2\|f(X_h^\star(n))\|^2
\\+2\langle X_h^\star(n),U_h(n+1)\rangle+\|U_h(n+1)\|^2,
\end{multline}
and therefore
\begin{multline} \label{eq.Xhnp12Xhn2norm}
\|X_h(n+1)\|^2\leq \|X_h(n)\|^2-2h\langle X_h^\star(n),f(X_h^\star(n))\rangle
\\+2\|X_h^\star(n)\|\|U_h(n+1)\|+\|U_h(n+1)\|^2.
\end{multline}
Suppose that $\Omega_5$ is the a.s. event such that $\Omega_5=\{\omega:\lim_{n\to\infty} \|U_h(n,\omega)\|=0\}$.
Clearly, we have that either the liminf of $\|X_h(n)\|$ is finite or not. Suppose that there exists
a nontrivial event $\Omega_6$ such that
\[
\Omega_6=\{\omega:\liminf_{n\to\infty}  \|X_h(n,\omega)\|<+\infty\}.
\]
In order to prove the result, it suffices to show that $\Omega_6$ is a.s. the same event as
$\{\omega:\liminf_{n\to\infty}  \|X_h(n,\omega)\|=+\infty\}$.

In order to do this, we suppose to the contrary that there exists an event $A=\{\omega\in \Omega_6: \liminf_{n\to \infty}\|X_h(n,\omega)\|=l(\omega)\in (0,\infty)\}$
for which $\mathbb{P}[A]>0$.  The finiteness of the liminf is a consequence of $A$ being a subset of $\Omega_6$. Let $A_1=A\cap \Omega_5$: then $\mathbb{P}[A_1]=\mathbb{P}[A]>0$. Fix $\omega\in A_1$. Suppose that $\liminf_{n\to \infty}\|X_h^\star(n,\omega)\|=0$.
Then, because $\|X_h^\star(n,\omega)\|\leq \|X_h(n,\omega)\|$ we have that $\liminf_{n\to \infty}\|X_h(n,\omega)\|=0$, a contradiction. Hence, for every $\omega\in A_1$ there exists $l^\star(\omega)>0$
such that $\liminf_{n\to \infty}\|X_h^\star(n,\omega)\|=l^\star(\omega)>0$.

Since $l(\omega)>0$, we note that $\varphi(l(\omega))>0$. Because for each $\omega\in A_1$ we have $U_h(n+1,\omega)\to 0$ as $n\to\infty$, it follows that for every $\omega\in A_1$
and for every
\[
\epsilon\in\left(0,1\wedge\frac{5l(\omega)}{2}\wedge h\frac{\varphi(l(\omega))}{5l(\omega)}\right),
\]
there is $N_1(\epsilon,\omega)\in \mathbb{N}$ such that
$\|U_h(n+1,\omega)\|<\epsilon$ for all $n>N_1(\epsilon,\omega)$. There also exists
$N_2(\omega) \in \mathbb{N}$ such that $\|X_h(n,\omega)\|>3l(\omega)/4$ for all
$n\geq N_2(\omega)$.

Now let
$N_3(\epsilon,\omega)=\max(N_1(\epsilon,\omega),N_2(\omega))$. By
the definition of the event $A\supseteq A_1$, it follows for each
$\omega\in A_1$ that there is a finite $N_4(\epsilon,\omega)$ such that
$N_4(\epsilon,\omega)=\inf \{n>N_3(\epsilon,\omega):
\|X_h(n,\omega)\|<5l(\omega)/4\}$.
Therefore $3l(\omega)/4<\|X_h(N_4,\omega)\|<5l(\omega)/4$.

We now show by induction that our supposition leads us to conclude that $3l(\omega)/4<\|X_h(n,\omega)\|<5l(\omega)/4$ for all $n\geq N_4(\epsilon,\omega)$.
This is certainly true for $n=N_4(\epsilon,\omega)$. Suppose that it is true for a
general $n\geq N_4(\epsilon,\omega)$. Clearly, as $n\geq N_4(\epsilon,\omega)>N_3(\epsilon,\omega)\geq N_2(\omega)$, we have $3l(\omega)/4<\|X_h(n+1,\omega)\|$, so it remains to establish the upper bound
$\|X_h(n+1,\omega)\|<5l(\omega)/4$.

Since $F_h$ is increasing, by using Lemmas~\ref{lem.Xastcontract} and
\ref{lemma.Finverse}, we get
\[
F_h^{-1}(3l(\omega)/4)<\|X_h^\star(n,\omega)\|\leq\|X_h(n,\omega)\|<\frac{5l(\omega)}{4}.
\]
Hence
\[
0<F_h^{-1}(3l(\omega)/4)<\|X_h^\star(n,\omega)\|<\frac{5l(\omega)}{4}.
\]
Since $\bar{f}$ is continuous, for all $y_2>y_1>0$, we have
\[
\inf_{y_1\leq \|x\|\leq y_2}\langle x,f(x)\rangle=\inf_{y\in[y_1,y_2]}\bar{f}(y)>0.
\]
Thus 
\[
\langle X_h^\star(n,\omega),f(X_h^\star(n,\omega))\rangle
\geq \min_{y\in[F_h^{-1}(3l(\omega)/4),5l(\omega)/4]}\bar{f}(y)=\varphi(l(\omega))>0.
\]

We now return to \eqref{eq.Xhnp12Xhn2norm} to estimate the terms on the righthand side.
For $\|X_h^\star(n,\omega)\|<5l(\omega)/4$, we have
\begin{multline*}
2\|X_h^\star(n,\omega)\|\|U_h(n+1,\omega)\|+\|U_h(n+1,\omega)\|^2\\
<2\frac{5l(\omega)}{4}\epsilon+\epsilon^2
<\frac{5l(\omega)}{2}\epsilon+\frac{5l(\omega)}{2}\epsilon
=5l(\omega)\epsilon.
\end{multline*}
Therefore
\begin{multline*}
-2h\langle X_h^\star(n,\omega),f(X_h^\star(n,\omega))\rangle
+2\|X_h^\star(N_4,\omega)\|\|U_h(n+1,\omega)\|+\|U_h(n+1,\omega)\|^2
\\
\leq -2h\varphi(l(\omega))+5l(\omega)\epsilon
<-2h\varphi(l(\omega))+5l(\omega)h\frac{\varphi(l(\omega))}{5l(\omega)} =-h\varphi(l(\omega)).
\end{multline*}
Therefore, by \eqref{eq.Xhnp12Xhn2norm}, we obtain $\|X_h(n+1,\omega)\|^2\leq
\|X_h(n,\omega)\|^2-h\varphi(l(\omega))$ and since by hypothesis we assume $\|X_h(n,\omega)\|<5l(\omega)/4$, we have
$\|X_h(n+1,\omega)\|<5l(\omega)/4$, as required. Moreover, scrutiny
of the above argument shows that one can equally prove that
\[
\|X_h(n+1,\omega)\|^2\leq \|X_h(n,\omega)\|^2-h\varphi(l(\omega)),
\quad \text{for all $n\geq N_4(\epsilon,\omega)$}.
\]
%
Therefore for any $N\in\mathbb{N}$ we have
\[
\|X_h(N_4+N,\omega)\|^2\leq \|X_h(N_4,\omega)\|^2-Nh\varphi(l(\omega)).
\]
In particular, let $N$ be any integer satisfying
\[
N>\frac{2}{h\varphi(l(\omega))}\left\{\left(\frac{5l(\omega)}{4}\right)^2-\left(\frac{l(\omega)}{4}\right)^2\right\}.
\]
Since $3l(\omega)/4<X_h(n,\omega)<5l(\omega)/4$ for all $n\geq N_4$, we get
\begin{align*}
\left(\frac{3l(\omega)}{4}\right)^2
\leq \|X_h(N_4+N,\omega)\|^2
&\leq \|X_h(N_4,\omega)\|^2-Nh\varphi(l(\omega))\\
&<\left(\frac{5l(\omega)}{4}\right)^2-Nh\varphi(l(\omega))\\
&<\left(\frac{l(\omega)}{4}\right)^2,
\end{align*}
which contradicts the original supposition. This proves the desired result.
\end{proof}
We are finally in a position to provide a proof of Theorem~\ref{thm:Xlim0}.
\subsection{Proof of Theorem~\ref{thm:Xlim0}}
To prove part (i), by virtue Lemma~\ref{lemma:Xliminf0}, it suffices to show on the event
$\Omega_7$  defined by $\Omega_7=\{\omega:\liminf_{n\to\infty}  \|X_h(n,\omega)\|=0\}$ (modulo some null event),
we have $X_h(n)\to 0$ as $n\to\infty$. We can assume, without loss of generality, that $\Omega_7$ is
an event of positive probability, because, if it is not, Lemma~\ref{lemma:Xliminf0} implies the event
$\{\lim_{n\to\infty} \|X_h(n)\|=+\infty\}$ is a.s., and our claim is trivially true.

Recall also the a.s. event $\Omega_5$ defined in Lemma~\ref{lemma:Xliminf0}, viz.,
\[
\Omega_5=\{\omega:\lim_{n\to\infty}  U_h(n,\omega)=0\}. \] By
Lemma~\ref{lemma:barfphi}, it follows that the function $\bar{f}$
defined in \eqref{def.barf} obeys $\bar{f}(y)>0$ for all $y>0$ and
by the continuity of $f$, $\bar{f}$ is also continuous on
$[0,\infty)$. Therefore, for any $l>0$ we have that
\begin{equation} \label{eq.Xto0barfineq}
\min_{\frac{l}{32}\leq y\leq \frac{l}{16}}\bar{f}(y)>0.
\end{equation}
Hence, we may choose an $\epsilon=\epsilon(l)>0$ so small that
\begin{equation} \label{eq.epsXto0}
2\epsilon(l)=1\wedge \frac{l}{32}   \wedge \left\{\frac{32}{10l}2h\min_{\frac{l}{32}\leq y\leq \frac{l}{16}}\bar{f}(y)\right\}.
\end{equation}
Let $\omega\in \Omega_8:=\Omega_5\cap \Omega_7$. Therefore, there exists $N_1(l,\omega)\in\mathbb{N}$ such that
$\|U_h(n+1,\omega)\|<\epsilon(l)$ for all $n>N_1(l,\omega)$. Moreover, as
$\liminf_{n\to\infty}  \|X_h(n,\omega)\|=0$, it follows that there exists an integer $N_2=N_2(l,\omega)>N_1(l,\omega)$
such that $\|X_h(N_2,\omega)\|<l/16$.


Suppose that there exists an integer $N_3>N_2$ such that $\|X_h(n,\omega)\|<l/16$ for $n=N_2,N_2+1,\ldots,N_3$, but $\|X_h(N_3+1,\omega)\|\geq l/16$. By \eqref{eq.SSupdate} and \eqref{def.U} we have
$X_h(N_3+1,\omega)=X_h^\star(N_3,\omega)+U_h(N_3+1,\omega)$, and since $N_3>N_1$, we obtain
\begin{align*}
\|X_h^\star(N_3,\omega)\|
&\geq \|X_h(N_3+1,\omega)\|-\|U_h(N_3+1,\omega)\|>\frac{l}{16}-\epsilon>\frac{l}{32},
\end{align*}
where \eqref{eq.epsXto0} is used at the last step.
Now using Lemma~\ref{lem.Xastcontract}, we get $\|X_h^\star(N_3)\|\leq \|X_h(N_3)\|<l/16$, and so
$l/32<\|X_h^\star(N_3)\|<l/16$. Therefore by the definition of $\bar{f}$, we have
\[
\langle X_h^\star(N_3),f(X_h^\star(N_3))\rangle\geq \min_{\frac{l}{32}\leq y\leq \frac{l}{16}}\bar{f}(y)>0,
\]
where the last inequality is a consequence of \eqref{eq.Xto0barfineq}.

We now insert these estimates into \eqref{eq.Xhnp12Xhn2norm} to get
\begin{align*}
\|X_h(N_3+1,\omega)\|^2&\leq \|X_h(N_3,\omega)\|^2-2h\langle X_h^\star(N_3,\omega),f(X_h^\star(N_3,\omega))\rangle\\
&\qquad\qquad+2\|X_h^\star(N_3,\omega)\|\|U_h(N_3+1,\omega)\|+\|U_h(N_3+1,\omega)\|^2\\
&\leq  \|X_h(N_3,\omega)\|^2-2h\langle X_h^\star(N_3,\omega),f(X_h^\star(N_3,\omega))\rangle\\
&\qquad\qquad
+2\|X_h^\star(N_3,\omega)\|\epsilon(l)+\epsilon(l)^2\\
&\leq \left(\frac{l}{16}\right)^2-2h\min_{\frac{l}{32}\leq y\leq \frac{l}{16}}\bar{f}(y)+2\frac{l}{16}\epsilon(l)+\epsilon(l)^2\\
&< \left(\frac{l}{16}\right)^2-2h\min_{\frac{l}{32}\leq y\leq \frac{l}{16}}\bar{f}(y)+2\frac{l}{16}\epsilon(l)+\epsilon(l)\frac{l}{32}\\
&=\left(\frac{l}{16}\right)^2-2h\min_{\frac{l}{32}\leq y\leq \frac{l}{16}}\bar{f}(y)
+\frac{5l}{32}\epsilon(l)\\
&<\left(\frac{l}{16}\right)^2,
\end{align*}
where once again \eqref{eq.Xto0barfineq} is used at the last step, and \eqref{eq.epsXto0} has been used throughout.
Therefore, by hypothesis we have
\[
\left(\frac{l}{16}\right)^2\leq \|X_h(N_3+1,\omega)\|^2 < \left(\frac{l}{16}\right)^2,
\]
%
%
%
a contradiction. Therefore, it must follow for each $\omega\in \Omega_8$ that for every $l>0$ there exists
an integer $N_2=N_2(l,\omega)$ such that $\|X_h(n,\omega)\|< l/16$ for all $n\geq N_2(l,\omega)$. Therefore,
we have that $X_h(n,\omega)\to 0$ as $n\to\infty$ for all $\omega\in \Omega_8$, and as $\Omega_8$ is a.s.,
the first part of the result has been proven.

To prove part (ii), define $A=\{\omega:\lim_{n\to\infty} X(n,\omega)=0\}$. Then $\mathbb{P}[A]>0$ by hypothesis. By Lemma~\ref{lem.Xastcontract}, we have that
$\|X_h^\star(n)\|\leq \|X_h(n)\|$ for all $n\geq 0$. Therefore, for $\omega\in A$, we have $X_h^\star(n,\omega)\to 0$ as $n\to\infty$.
By \eqref{eq.SSupdate}, we have that
\[
\lim_{n\to\infty} U_h(n+1,\omega)=\lim_{n\to\infty} \left\{X_h(n+1,\omega)-X_h^\star(n,\omega)\right\}=0.
\]
Therefore $U_h(n)\to 0$ on a set of positive probability. By
Lemma~\ref{lemma.Uasy}, it follows that $S_h(\epsilon)<+\infty$ for all $\epsilon>0$.

\subsection{Proof of Theorem~\ref{thm:Xlim0faway0}}
Scrutiny of Theorem~\ref{thm:Xlim0} shows that we can establish Theorem~\ref{thm:Xlim0faway0} provided that
the condition \eqref{eq.xfxunifpos} together with $S_h(\epsilon)$ always being finite implies
$\liminf_{n\to\infty} \|X_h(n)\|<+\infty$ a.s. This is the subject of the next result. 
\begin{lemma}  \label{lemma.Xliminffinite}
Suppose that $(X_h,X^\star_h)$ is a solution of
\eqref{eq.splitstep}. Suppose that $f$ obeys
\eqref{eq.fglobalunperturbed} and \eqref{eq.xfxunifpos} and that the
sequence $\xi$ obeys Assumption~\ref{ass.normal}. If $S_h(\epsilon)$
defined by \eqref{def.Seps} obeys $S_h(\epsilon)<+\infty$ for every
$\epsilon>0$, then
\[
\liminf_{n\to\infty} \|X_h(n)\|<+\infty, \quad \text{a.s.}
\]
\end{lemma}
\begin{proof}
Suppose to the contrary that
\[
A=\{\omega: \liminf_{n\to\infty} \|X(n,\omega)\|=+\infty\} \]
is an
event with $\mathbb{P}[A]>0$. Since $\Omega_1=\{\omega:
U_h(n,\omega)\to 0 \text{ as } n\to\infty\}$ is an a.s. event, we
have that $A_1=A\cap \Omega_1$ obeys $\mathbb{P}[A_1]>0$. Therefore
by \eqref{eq.xfxunifpos} for each $\omega\in A_1$, there is an
$N(\omega)\in \mathbb{N}$ such that
\[
\langle X_h^\star(n-1,\omega),f(X_h^\star(n-1,\omega))\rangle\geq \frac{\phi}{2}, \quad n\geq N_1(\omega).
\]
On the other hand, as $U_h(n,\omega)\to 0$ as $n\to\infty$ for each $\omega\in A_1$, we have that there is $N_2(\omega)$ such that
\[
\|U_h(n,\omega)\|^2<h\frac{\phi}{4}, \quad n\geq N_2(\omega).
\]
Suppose $N_3(\omega)=\max(N_1(\omega),N_2(\omega))$. Then
by Lemma~\ref{lemma.repX2}, we have that $\|X_h\|^2$ obeys
\begin{multline*} 
\|X_h(n)\|^2=\|X_h(N_3)\|^2 - \sum_{i=N_3+1}^n h\left\{2\langle f(X_h^\star(i-1)), X_h^\star(i-1)\rangle - \frac{1}{h}\|U_h(i)\|^2 \right\}
\\- \sum_{i=N_3+1}^n h^2\|f(X_h^\star(i-1))\|^2+ M(n)-M(N_3), \quad n\geq N_3+1.
\end{multline*}
Since for $n\geq N_3(\omega)$ we have
\[
2\langle X_h^\star(n-1,\omega),f(X_h^\star(n-1,\omega))\rangle -\frac{1}{h}\|U_h(n,\omega)\|^2  > \frac{3\phi}{4},
\]
we get
\begin{multline} \label{eq.X2sum2}
\|X_h(n,\omega)\|^2\leq \|X_h(N_3(\omega),\omega)\|^2 - \frac{3\phi h}{4}(n-N_3(\omega))+ M(n,\omega)-M(N_3(\omega),\omega), \\
 n\geq N_3(\omega)+1, \quad \omega\in A_1.
\end{multline}
Now, recall that $M$ is defined by \eqref{def.M2} where $Y^{(j)}$ is given by \eqref{def.Yi} for $j=1,\ldots,r$. Notice by  \eqref{def.Yi} that $Y^{(j)}(n)$ is an $\mathcal{F}^\xi(n)$--measurable random variable. Since $\xi$ obeys Assumption~\ref{ass.normal}, it follows that
all the conditions of Lemma~\ref{lemma.recurrentmartingale} hold, and that the martingale $M$ is in the form of \eqref{eq.M} in
Lemma~\ref{lemma.recurrentmartingale}. Therefore, it follows that $M$ obeys \eqref{eq.Mlimitoroscillate}, so that, if we define
\[
\Omega_l=\{\omega:\lim_{n\to\infty} M(n,\omega)\text{ exists and is finite}\}
 \]
 and
 \[
 \Omega_\infty=\{\omega:\liminf_{n\to\infty} M(n,\omega)=-\infty, \quad \limsup_{n\to\infty} M(n,\omega)=+\infty\}
 \]
 then $\Omega_l\cup \Omega_\infty=:\Omega_2$ is an a.s. event. 
Since $\Omega_2$ is a.s., it follows that  either (or both) of $A_2:=A_1\cap \Omega_l$ and $A_3:=A_1\cap \Omega_\infty$ are events of positive probability.

Suppose that $\mathbb{P}[A_2]>0$. Then, for each $\omega\in A_2$ we have that $M(n,\omega)$ has a finite limit (say $L(\omega)$) as $n\to\infty$, and that $\|X_h(n,\omega)\|\to\infty$ as $n\to\infty$. Taking the liminf as $n\to\infty$ on both sides of \eqref{eq.X2sum2} gives
\begin{align*} 
+\infty &=
\liminf_{n\to\infty}\|X_h(n,\omega)\|^2\\
&\leq \|X_h(N_3(\omega),\omega)\|^2 -M(N_3(\omega),\omega)+
\liminf_{n\to\infty} \left\{ - \frac{3\phi h}{4}(n-N_3(\omega))+ M(n,\omega)\right\}\\
&=-\infty,
\end{align*}
a contradiction. Therefore, we have $\mathbb{P}[A_2]=0$.

Suppose now that $\mathbb{P}[A_3]>0$. Then, for each $\omega\in A_3$ it follows from the definition of $A_3$ that $\liminf_{n\to\infty}M(n,\omega)=-\infty$, and that $|X_h(n,\omega)|\to\infty$ as $n\to\infty$. Taking the liminf as $n\to\infty$ on both sides of \eqref{eq.X2sum2} gives
\begin{align*} 
+\infty &=
\liminf_{n\to\infty}\|X_h(n,\omega)\|^2\\
&\leq \|X_h(N_3(\omega),\omega)\|^2 -M(N_3(\omega),\omega)+
\liminf_{n\to\infty} \left\{ - \frac{3\phi h}{4}(n-N_3(\omega))+ M(n,\omega)\right\}\\
&=-\infty,
\end{align*}
a contradiction. Therefore, we have $\mathbb{P}[A_3]=0$. Therefore, we have that $0=\mathbb{P}[A_2\cup A_3]=\mathbb{P}[A_1\cap \Omega_2]>0$,
a contradiction. Hence $\mathbb{P}[A_1]=0$, and so $\mathbb{P}[A]=0$, which proves the result.
\end{proof}

\subsection{Proof of Theorem~\ref{theorem.xto0}}
To prove this, we first consider the case when $\sigma_h\in \ell^2(\mathbb{N})$. In this case, Theorem~\ref{thm.sig2l1Xto0} implies that 
$\lim_{n\to \infty}X_h(n)=0$, a.s. Therefore, we concentrate next on the case when $\sigma_h\notin \ell^2$. 
An important step to achieve this is to prove the following lemma. 
\begin{lemma} \label{lemma.liminf0discrete}
Suppose that $f$ obeys \eqref{eq.fglobalunperturbed} and $(X_h,X^*_h)$ is a solution of \eqref{eq.splitstep}. Suppose also that 
$\sigma_h \notin l^2(\mathbb{N})$. Then
\[
\liminf_{n\to \infty}X_h(n)\leq 0\leq\limsup_{n\to \infty}X_h(n), \quad \text{a.s.}
\]
\end{lemma}
\begin{proof}
Suppose $\liminf_{n\to\infty}X_h(n)>0$ with positive probability. Then there exists an event $A$ with $\mathbb{P}[A]>0$, such that 
\[
A=\{\omega:\liminf_{n\to\infty}X_h(n,\omega)=\underline{X}(\omega)>0\}.
\]
For $\omega\in A$, set $\underline{X}(\omega):=\liminf_{n\to\infty}X_h(n,\omega)>0$. Suppose $\liminf_{n\to\infty} X^\star_h(n,\omega)=0$, so that 
there exists a sequence $(n_j(\omega))_{j=1}^\infty$ such that $n_j(\omega)\uparrow\infty$ as $j\to\infty$ such that $\lim_{j\to\infty} X^\star_h(n_j(\omega),\omega)=0$. Therefore, as $X_h(n,\omega)=X_h^\star(n,\omega)+hf(X_h^\star(n,\omega))$, we have that 
\begin{multline*}
0<\underline{X}(\omega)=\liminf_{n\to\infty} X(n,\omega)\leq 
\lim_{j\to\infty} X_h(n_j(\omega),\omega)\\
=\lim_{j\to\infty}\left\{X_h^\star(n_j(\omega),\omega)+hf(X_h^\star(n_j(\omega),\omega))\right\}=0,
\end{multline*}
a contradiction. Hence for each $\omega\in A$ we have that 
\[
\liminf_{n\to\infty}X^\star_h(n,\omega)=:\underline{X}^*(\omega)>0.
\]
Therefore, for each $\omega\in A$, there is $N^*(\omega)\in\mathbb{N}$ such that
$X_h(n,\omega)\geq \underline{X}(\omega)/2$ and $X^\star_h(n,\omega)\geq \underline{X}^*(\omega)/2$ for all $n\geq N^*(\omega)$. 
Let $n\geq N^\ast(\omega)$. Since 
\[
X_h(n+1)=X^\star_h(n)+\sqrt{h}\sigma_h(n)\xi(n+1) =X_h(n)-hf(X^\star_h(n))+\sqrt{h}\sigma_h(n)\xi(n+1),
\]
we have 
\begin{align*}
X_h(n+1,\omega)
&=X_h(N^*(\omega),\omega)-h\sum^n_{j=N^*(\omega)}f((X^*_h(j,\omega))+\sum^n_{j=0}\sqrt{h}\sigma_h(j)\xi(j+1,\omega)\\
&\qquad-\sum^{N^*(\omega)-1}_{j=0}\sqrt{h}\sigma_h(j)\xi(j+1,\omega)\\
&\leq X_h(N^*(\omega),\omega)-\sum^{N^*(\omega)-1}_{j=0}\sqrt{h}\sigma_h(j)\xi(j+1,\omega)+M_h(n+1),
\end{align*}
where we have defined the martingale $M_h$ by 
\[
M_h(n+1)=\sum_{j=0}^n \sqrt{h}\sigma_h(j)\xi(j+1), \quad n\geq 0.
\]
Since $\sigma_h\notin \ell^2(\mathbb{N})$, we have that for a.a. $\omega\in A$,
$\liminf_{n\to \infty}M_h(n+1,\omega)=-\infty$. Therefore, we have 
\[
0<\liminf_{n\to\infty} X_h(n+1,\omega)\leq -\infty \quad\text{for a.a. $\omega\in A$},
\]
a contradiction. Therefore $\mathbb{P}[A]=0$, so $\liminf_{n\to\infty} X_h(n)\leq 0$, a.s. 
One can proceed analogously to prove that $\limsup_{n\to\infty} X_h(n)\geq 0$ a.s.
\end{proof}

\begin{proof}[Proof of Theorem~\ref{theorem.xto0}]
 Define 
\[
A_1=\{\omega:\lim_{n\to\infty}|X_h(n,\omega)|=+\infty \},\quad A_0=\{\omega:\lim_{n\to \infty}X_h(n,\omega)=0\}.
\]
Note that Theorem~\ref{thm:Xlim0} and the hypothesis $S_h(\epsilon)<+\infty$ implies that $\Omega^*=A_1\cup A_0$ is an a.s. event. 
Suppose $A_1$ is an event with positive probability. 
Let 
\[
\Omega_1=\{\omega: \liminf_{n\to\infty}X_h(n,\omega)\leq 0,\quad \limsup_{n\to\infty}X_h(n,\omega)\}\geq 0\}
\] 
and $\Omega_2=\{\omega:\lim_{n\to\infty}\sqrt{h}\sigma_h(n)\xi(n+1,\omega)=0\}$. 
By Lemma~\ref{lemma.liminf0discrete}, $\Omega_1$ is an a.s. event, and $S_h(\epsilon)<+\infty$ for all $\epsilon>0$ implies that $\Omega_2$ is 
an a.s. event. 
Define $A_2=A_1\cup\Omega_1\cup\Omega_2$. Then $\mathbb{P}[A_2]=\mathbb{P}[A_1]>0$. 

Next, let $\epsilon\in (0,1/2)$. Then for every $\omega\in A_2$, there exists an $N_0(\omega,\epsilon)$ such that for all $n\geq N_0(\omega,\epsilon)$ 
we have $|\sqrt{h}\sigma_h(n)\xi(n+1,\omega)|<\epsilon$ and $|X_n(n,\omega)|>1/\epsilon$. 
Since $\lim_{n\to\infty} |X_h(n,\omega)|=+\infty$, $\liminf_{n\to\infty} X_h(n,\omega)\leq 0$ and $\limsup_{n\to\infty} X(n,\omega)\geq 0$, we must have 
\[
\liminf_{n\to\infty}X_h(n,\omega)=-\infty,\quad \limsup_{n\to\infty}X_h(n,\omega)=+\infty.
\]
Therefore as $\lim_{n\to\infty}|X_h(n,\omega)|=+\infty$, it follows that there exists $N^*(\omega,\epsilon)>N_0(\omega,\epsilon)$ such that  
\begin{equation*}
X_h(N^*(\omega,\epsilon),\omega)<-\frac{1}{\epsilon},\quad X_h(N^*(\omega,\epsilon)+1,\omega)>\frac{1}{\epsilon}.
\end{equation*}
Therefore
\begin{align*}
\frac{1}{\epsilon}<X_h(N^*(\omega,\epsilon)+1,\omega)&=X_h^*(N(\omega,\epsilon),\omega)+\sqrt{h}\sigma_h(n)\xi(N(\omega,\epsilon),\omega)\\
&\leq X^*_h(N(\omega,\epsilon),\omega)+\epsilon.
\end{align*}
Finally, because $X_h(N^*(\omega,\epsilon),\omega)<-1/\epsilon<0$, we have that
$X_h(N^*(\omega,\epsilon),\omega)\leq X^\star_h(N^*(\omega,\epsilon),\omega)\leq 0$. 
Therefore 
\[
\frac{1}{\epsilon}\leq X^\star_h(N(\omega,\epsilon),\omega)+\epsilon\leq \epsilon.
\]
Hence $\epsilon^2\geq 1$. But $\epsilon\in(0,1/2)$, which is a contradiction. Therefore $\mathbb{P}[A_1]=0$ and so as $A_0$ and $A_1$ 
are disjoint events we have 
\[
1=\mathbb{P}[\Omega^*]=\mathbb{P}[A_1\cup A_0]=\mathbb{P}[A_1]+\mathbb{P}[A_0]=\mathbb{P}[A_0].
\]
Thus $A_0=\{\omega:\lim_{n\to\infty}X_h(n,\omega)=0\}$ is an a.s. event, which finishes the proof.
\end{proof}

\section{Proof of Theorem~\ref{lemma.Xboundedabove}}
\subsection{Proof of parts (C), (A), and limsup in part (B)}
Part (C) of the Theorem follows from part (A) of Theorem~\ref{theorem.XunboundedXboundedbelow}.
Part (A) is a consequence of Theorem~\ref{thm:Xlim0faway0}, because the condition \eqref{eq.xfxunifpos} on $f$ is implied by
\eqref{eq.fboundedbelow}. The lower bound in part (B) is a consequence of part (B) of Theorem~\ref{theorem.XunboundedXboundedbelow}.
Hence the result holds if we can establish the upper bound in part (B).

To do this, notice first by part (B) of Lemma~\ref{lemma.Uasy} that there exists an a.s. event $\Omega_1$ given by $\Omega_1=\{\omega:\limsup_{n\to\infty} \|U_h(n,\omega)\|_1\leq c_2\}$, where $c_2$ is given by \eqref{eq.Ubounds}. Therefore, there is a deterministic $B_0>c_2$ such that for each
$\omega\in \Omega_1$ there is an $N=N(\omega)\in \mathbb{N}$ such that $\|U_h(n+1,\omega)\|_2\leq\|U_h(n+1,\omega)\|_1\leq  B_0$
for all $n\geq N$. Since $f$ obeys \eqref{eq.fboundedbelow}, we may define
\[
M(B_0)=\sup\{y>0: \inf_{\|x\|_2\geq y} \frac{\langle x,f(x)\rangle}{\|x\|_2}\leq \frac{2B_0}{h}\}.
\]
Define $C(B_0)=B_0+M(B_0)$. Now suppose that $\|X_h(n,\omega)\|_2>C(B_0)$ for all $n\geq N(\omega)$.
Let $n\geq N(\omega)$. By \eqref{eq.SSupdate} and \eqref{def.U}, we have
$\|X_h^\star(n,\omega)\|_2\geq \|X_h(n+1,\omega)\|_2-\|U_h(n+1,\omega)\|_2\geq C(B_0)-B_0=M(B_0)$.
Hence by the definition of $M(B_0)$ we have
\[
\frac{\langle X_h^\star(n,\omega),f(X_h^\star(n,\omega))\rangle}{\|X_h^\star(n,\omega)\|_2}\geq \frac{2B_0}{h}.
\]
Therefore by  \eqref{eq.SSXast} we get
\begin{align*}
\langle X^\star_h(n,\omega),X_h(n,\omega)\rangle&=\|X_h^\star(n,\omega)\|^2_2+h\langle f(X_h^\star(n,\omega)),X_h^\star(n,\omega)\rangle\\
&\geq \|X_h^\star(n,\omega)\|^2_2+h\frac{2B_0}{h}\|X_h^\star(n,\omega)\|_2.
\end{align*}
By the Cauchy--Schwartz inequality,
\[
\|X_h^\star(n,\omega)\|_2\|X_h(n,\omega)\|_2\geq \|X_h^\star(n,\omega)\|^2_2+2B_0\|X_h^\star(n,\omega)\|_2.
 \]
Since $\|X_h^\star(n,\omega)\|>0$, we have $\|X_h(n,\omega)\|_2\geq \|X_h^\star(n,\omega)\|_2+2B_0$, or
 $ \|X_h^\star(n,\omega)\|_2\leq \|X_h(n,\omega)\|_2-2B_0$. Therefore, for $n\geq N$ by \eqref{eq.SSupdate} we have
\[
\|X_h(n+1,\omega)\|_2 
\leq \|X_h^\star(n,\omega)\|+B_0
\leq \|X_h(n,\omega)\|_2-B_0.
\]
Therefore, we have
\[
C(B_0)\leq \|X_h(N+n,\omega)\|_2\leq \|X_h(N,\omega)\|_2-B_0n, \quad n\geq 0,
\]
which is a contradiction. Thus, there exists $N_1=N_1(\omega)\geq N(\omega)$ such that $\|X_h(N_1)\|_2\leq C(B_0)$.

We prove by induction that $\|X_h(n)\|_2\leq C(B_0)$ for all $n\geq N_1$. Suppose that this is true at level $n$.
Suppose that $\|X_h^\star(n,\omega)\|>C(B_0)-B_0$. Now by \eqref{eq.SSXast} we get
\begin{align*}
\langle X^\star_h(n,\omega),X_h(n,\omega)\rangle&=\|X_h^\star(n,\omega)\|^2+h\langle f(X_h^\star(n,\omega)),X_h^\star(n,\omega)\rangle\\
&\geq \|X_h^\star(n,\omega)\|^2+h\frac{2B_0}{h}\|X_h^\star(n,\omega)\|.
\end{align*}
By the Cauchy--Schwartz inequality,
\[
\|X_h^\star(n,\omega)\|_2\|X_h(n,\omega)\|_2\geq \|X_h^\star(n,\omega)\|^2_2+2B_0\|X_h^\star(n,\omega)\|_2.
 \]
Since $\|X_h^\star(n,\omega)\|>0$, we have $\|X_h(n,\omega)\|_2\geq \|X_h^\star(n,\omega)\|_2+2B_0>C(B_0)+B_0$. But
$C(B_0)\geq \|X_h(n,\omega)\|_2>C(B_0)+B_0$, a contradiction. Hence $\|X_h^\star(n,\omega)\|\leq C(B_0)-B_0$. Therefore by
\eqref{eq.SSupdate}, we have
\[
\|X_h(n+1,\omega)\|_2\leq \|X_h^\star(n,\omega)\|_2 + B_0\leq C(B_0),
\]
which proves the claim at level $n+1$. Therefore we have $\|X_h(n,\omega)\|_2\leq C(B_0)$
for all $n\geq N_1(\omega)$ and all $\omega\in \Omega_1$, which is an a.s. event.
Hence $\limsup_{n\to\infty}\|X_h(n,\omega)\|_2\leq C(B_0)$ for each $\omega\in \Omega_1$.
Therefore, we have   $\limsup_{n\to\infty}\|X_h(n)\|_2\leq c_4$ a.s., where $c_4:=C(B_0)$ is deterministic.

\subsection{Proof of liminf in part (B)}
It remains to prove in the following result.
\begin{lemma} \label{lemma.liminfXdisc}
Suppose that $S_h(\epsilon)<+\infty$ for all $\epsilon>\epsilon'$ and $S_h(\epsilon)=+\infty$ for all $\epsilon<\epsilon'$. Then  
\[
\liminf_{n\to\infty} \|X_h(n)\|=0, \quad \text{a.s.}
\]
\end{lemma}
In order to do this we need first a technical lemma.
\begin{lemma} \label{lemma.sigmahxiLLN}
$S_h(\epsilon)<+\infty$ for all $\epsilon>\epsilon'$ and $S_h(\epsilon)=+\infty$ for all $\epsilon<\epsilon'$. Then 
\begin{equation}  \label{eq.sigmanhto0}
\lim_{n\to\infty} \|\sigma_h(n)\|_F=0,
\end{equation} 
and 
\begin{equation} \label{eq.sigmahxiLLN}
\lim_{n\to\infty} \frac{1}{n}\sum_{j=0}^{n-1} \|\sigma_h(j-1)\xi(j)\|^2=0, \quad \text{a.s.}
\end{equation}
\end{lemma}
\begin{proof}
First, we note that if $S_h(\epsilon)<+\infty$ for some $\epsilon>0$, it follows that 
\[
1-\Phi\left(\frac{\epsilon}{h\|\sigma_h(n)\|_F}\right)\to 0, \text{ as $n\to\infty$}.
\]
and therefore \eqref{eq.sigmanhto0} holds. 
Define 
\[
\beta(n)=\|\sigma_h(n-1)\xi(n)\|^2, \quad n\geq 1.
\]
Notice that the independence of $\xi(n)$ imply that $(\beta(n))_{n\geq 1}$ is a sequence of independent random variables. 
Using \eqref{eq.varsigxi}, we have that 
\begin{equation*} 
\mathbb{E}[\beta(n)]=\mathbb{E}[\|\sigma_h(n-1)\xi(n)\|^2]= \|\sigma_h(n-1)\|_F^2, \quad n\geq 1.
\end{equation*}
Notice from \eqref{eq.sigmanhto0} that $\mathbb{E}[\beta(n)]\to 0$ as $n\to\infty$. 
Define $\tilde{\beta}(n)=\beta(n)-\mathbb{E}[\beta(n)]$ for $n\geq 1$. Then  $(\tilde{\beta}(n))_{n\geq 1}$ is a sequence of independent zero mean 
random variables. We will presently show that 
\begin{equation} \label{eq.4thmombeta0}
\lim_{n\to\infty} \mathbb{E}[\beta(n)^4]= 0.
\end{equation}
Taken together with $\mathbb{E}[\beta(n)]\to 0$ as $n\to\infty$, we see that $\lim_{n\to\infty} \mathbb{E}[\tilde{\beta}(n)^4]= 0$, so that 
there exists a constant $K>0$ for which $\mathbb{E}[\tilde{\beta}(n)^4]\leq K$ for all $n\geq 0$. Therefore, by this estimate, and the fact that 
 $(\tilde{\beta}(n))_{n\geq 1}$ is a sequence of independent zero mean 
random variables, the version of the strong law of large numbers appearing in Theorem 7.2 in \cite{Williams:1991}, enables us to conclude that
\[
\lim_{n\to\infty} \frac{1}{n}\sum_{j=1}^n \tilde{\beta}(j)=0, \quad\text{a.s.}
\]
Since $\mathbb{E}[\beta(n)]\to 0$ as $n\to\infty$, we have that 
\[
\lim_{n\to\infty} \frac{1}{n}\sum_{j=1}^n \beta(j)=0, \quad\text{a.s.}
\]
which is precisely \eqref{eq.sigmahxiLLN}.

It remains to prove \eqref{eq.4thmombeta0}. Since $\|Ax\|_2\leq \|A\|_F\|x\|_2$ for any $x\in\mathbb{R}^r$ and $A\in \mathbb{R}^{d\times r}$, we have that 
\begin{align*}
\mathbb{E}[\beta(n)^4]&=\mathbb{E}[\|\sigma_h(n-1)\xi(n)\|_2^8]\leq \mathbb{E}[\|\sigma_h(n-1)\|_F^8\|\xi(n)\|_2^8]\\
&=\|\sigma_h(n-1)\|_F^8\mathbb{E}[\|\xi(n)\|_2^8].
\end{align*}
Since $(\xi(n))_{n\geq 1}$ are identically and distributed Gaussian vectors with independent entries (each of which is a standard normal random variable), we have that there is $K_1:=\mathbb{E}[\|\xi(n)\|_2^8]$ for all $n\geq 1$. Hence $\mathbb{E}[\beta(n)^4]\leq K_1\|\sigma_h(n-1)\|_F^8$ for $n\geq 1$. 
Since \eqref{eq.sigmanhto0} holds, we have that $\mathbb{E}[\beta(n)^4]\to 0$ as $n\to\infty$, as claimed. 
\end{proof}

\subsection{Proof of Lemma~\ref{lemma.liminfXdisc}}
Recall the representation of $\|X_h\|^2$ in \eqref{eq.X2sum} i.e.,
\begin{multline} \label{eq.X2sum3}
\|X_h(n)\|^2=\|X_h(0)\|^2 - 2\sum_{i=1}^n h\langle f(X_h^\star(i-1)), X_h^\star(i-1)\rangle +\sum_{i=1}^n h\|\sigma_h(i-1)\xi(i)\|^2
\\- \sum_{i=1}^n h^2\|f(X_h^\star(i-1))\|^2+ M(n), \quad n\geq 1,
\end{multline}
where the martingale $M$ defined by \eqref{def.Yi} and \eqref{def.M2} i.e., 
\begin{gather*} 
Y^{(j)}(n)= 2\sqrt{h} \sum_{k=1}^d [X_h^\star(n)]_k  [\sigma_h(n)]_{kj}, \quad j=1,\ldots,r, \quad n\geq 1,\\
M(n)=\sum_{i=1}^n \sum_{j=1}^r Y^{(j)}(i-1)\xi^{(j)}(i), \quad n\geq 1.
\end{gather*}
Then $M$ has quadratic variation estimated by \eqref{eq.estsqvM} i.e.,
\[
\langle M\rangle(n) \leq 4h\sum_{j=0}^{n-1}\|X_h^\star(j)\|^2\|\sigma_h(j)\|^2_F.
\]
Since $\|X_h^\ast(n)\|$ is a bounded sequence, and $\|\sigma_h(n)\|_F\to 0$ as $n\to\infty$, we have that 
\[
\lim_{n\to\infty} \frac{1}{n}\langle M\rangle(n)=0, \quad \text{a.s.}
\]
Suppose that $A_1=\{\omega: \lim_{n\to\infty} \langle M\rangle(n,\omega)=+\infty\}$. Then by the Law of Large numbers for martingales, we have 
\[
\lim_{n\to\infty} \frac{1}{n}M(n,\omega)=\lim_{n\to\infty}\frac{M(n,\omega)}{\langle M\rangle(n,\omega)}\cdot \frac{\langle M\rangle(n,\omega)}{n}=0,
\]
for a.a. $\omega\in A_1$. Suppose that $A_2=\{\omega: \lim_{n\to\infty} \langle M\rangle(n,\omega)<+\infty\}$. Then by the martingale convergence theorem
we have that $\lim_{n\to\infty} M(n,\omega)$ is finite for a.a. $\omega\in A_2$, so we automatically have $\lim_{n\to\infty} M(n,\omega)/n=0$ for a.a. $\omega\in A_2$. Therefore we have that 
\begin{equation}\label{eq.Mndivnto0}
\lim_{n\to\infty} \frac{M(n)}{n}=0, \quad\text{a.s.}
\end{equation}
By Lemma~\ref{lemma.liminfXdisc} we have that 
\[
\lim_{n\to\infty} \frac{1}{n}\sum_{i=1}^n h\|\sigma_h(i-1)\xi(i)\|^2=0, \quad\text{a.s.}
\]
Recalling that $n\mapsto \|X_h(n)\|$ is a.s. bounded, we can use the last limit, \eqref{eq.Mndivnto0} and \eqref{eq.X2sum3} to obtain
\begin{equation} \label{eq.RXstarave0}
\lim_{n\to\infty} \frac{1}{n}\sum_{i=1}^n hR(X_h^\star(i-1))=0, \quad\text{a.s.} 
\end{equation} 
recalling the definition of $R$ from \eqref{def.R}. 

Next, we suppose that $A$ defined by 
\[
A=\{\omega:\liminf_{n\to\infty} \|X_h(n,\omega)\|>0\}.
\]
is such that $\mathbb{P}[A]>0$. Let $\Omega_1=\{\omega: \limsup_{n\to\infty} \|X_h(n,\omega)\|<+\infty\}$ and $A_1=A\cap \Omega_1$. Then $\mathbb{P}[A_1]=\mathbb{P}[A]>0$. Then for each $\omega\in A_1$, $\liminf_{n\to\infty} \|X_h^\star(n,\omega)\|>0$. Therefore, using the fact that $\|X_h^\star(n)\|\leq \|X_h(n)\|$, we see 
that $\|X_h^\star(n,\omega)\|$ is bounded for $\omega\in A_1$ and therefore, for every $\omega\in A_1$ there is an $N(\omega)\in \mathbb{N}$ and 
$0<\underline{X}_h(\omega)\leq \overline{X}_h(\omega)<+\infty$ such that 
\[
\frac{1}{2}\underline{X}_h(\omega)\leq \|X_h^\star(n,\omega)\|\leq 2\overline{X}_h(\omega), \quad n\geq N(\omega).
\] 
Now, we recall that $R:\mathbb{R}^d\to \mathbb{R}$ defined by \eqref{def.R} is continuous and obeys $R(x)>0$ for all $x\neq 0$ and $R(0)=0$.
Therefore, for any $0<a\leq b<+\infty$, we have 
\[
\inf_{a\leq \|x\|\leq b} R(x)=:L_h(a,b)>0.
\] 
Therefore, for all $n\geq N(\omega)$ we have 
\[
R(\|X_h^\star(n,\omega)\|)\geq L_h\left(\frac{1}{2}\underline{X}_h(\omega),2\overline{X}_h(\omega)\right)=:\lambda_h(\omega)>0.
\]
Hence, as $R(x)\geq 0$ for all $x\geq 0$, we have for $n\geq N(\omega)+1$ that
\begin{align*}
\frac{1}{n}\sum_{i=1}^n hR(X_h^\star(i-1,\omega))
&\geq \frac{1}{n}\sum_{i=N(\omega)+1}^n hR(X_h^\star(i-1,\omega))\\
&\geq  \frac{1}{n}\sum_{i=N(\omega)+1}^n h \lambda_h(\omega) 
=\frac{1}{n} (n-N(\omega))  h\lambda_h(\omega).
\end{align*}
Therefore, we have for each $\omega\in A_1$  
\[
\liminf_{n\to\infty} \frac{1}{n}\sum_{i=1}^n hR(X_h^\star(i-1,\omega)) \geq  h\lambda_h(\omega)>0,
\]
or 
\[
\liminf_{n\to\infty} \frac{1}{n}\sum_{i=1}^n hR(X_h^\star(i-1)) >0, \quad \text{on $A_1$}. 
\]
Since $\mathbb{P}[A_1]>0$ this contradicts \eqref{eq.RXstarave0}, and so we must have $\mathbb{P}[A_1]=0$. Hence we have that 
$\liminf_{n\to\infty} \|X_h(n)\|=0$ a.s. as claimed. 

\section{Proof of Theorem~\ref{thm.linsplit}}
We prove the result in two parts. First, we prove everything apart from the limit inferior in part (B), and then show that 
\[
\liminf_{n\to\infty} \|Y_h(n)\|=0, \quad \text{a.s.}
\]
in case (B), when the solution has already been shown to be bounded.
\subsection{Proof of Theorem~\ref{thm.linsplit} apart from liminf in part (B)}
Part (C) is a direct consequence of part (A) of Theorem~\ref{theorem.XunboundedXboundedbelow}. The
lower bound in part (B) is an automatic consequence of part (B) of Theorem~\ref{theorem.XunboundedXboundedbelow}.

It remains to prove part (A) and the upper bound in part (B). We start by determining the eigenvalues of $C(h)$.
If $c_{C(h)}$ be the characteristic polynomial of $C(h)$, then we have $c_{C(h)}(0)=(-1)^d \det(C(h))\neq 0$ and
\[
c_{C(h)}(\lambda)=\frac{1}{\det(I-Ah)}(\lambda h)^d c_A\left(\frac{\lambda -1}{\lambda h}\right), \quad \lambda\neq 0.
\]
Therefore, $\lambda_A$ is an eigenvalue of $A$ if and only if $\lambda_h=1/(1-\lambda_A h)$ is an eigenvalue of $C(h)$.
Since \eqref{eq.linstab} holds, $0$ is not an eigenvalue of $A$, and for every $h>0$,
\[
\text{Re}(\lambda_A)<0<\frac{h}{2}|\lambda_A|^2.
\]
This implies that $|1-h\lambda_A|<1$, and hence that $|\lambda_h|<1$ for each eigenvalue of $C(h)$.
$Y_h$ obeys
\[
Y_h(n)=C(h)^n \zeta+ \sum_{j=1}^n C(h)^{n-j}U_h(j), \quad n\geq 0.
\]

For part (A), if $S_h(\epsilon)<+\infty$ for every $\epsilon>0$, by Lemma~\ref{lemma.Uasy}, we have that $U_h(n)\to 0$ as $n\to\infty$.
Since all eigenvalues of $C(h)$ are less than unity in modulus, it follows that $\sum_{j=1}^n C(h)^{n-j}U_h(j)\to 0$  as $n\to\infty$,
proving the result. To prove the upper bound in part (B), we note that for every $\epsilon\in (0,(1-\rho(C(h)))/2)$, there is a norm
$\|\cdot\|_N$ such that
\[
\|C(h)^k x\|_N\leq \|C(h)^k\|_N \|x\|_N \leq (\rho(C(h))+\epsilon)^k \|x\|_N \text{ for all $k\geq 0$ and all $x\in \mathbb{R}^d$.}
\]
Hence we have
\[
\|Y_h(n)\|_N\leq (\rho(C(h))+\epsilon)^n \|\zeta\|_N  + \sum_{j=1}^n   (\rho(C(h))+\epsilon)^{n-j} \|U_h(j)\|_N.
\]
Therefore taking limits and using the fact that there is a $c>0$ such that $\|x\|_N\leq c\|x\|_1$ for all $x\in\mathbb{R}^d$, we obtain
\[
\limsup_{n\to\infty} \|Y_h(n)\|_N\leq  \frac{1}{1-(\rho(C(h))+\epsilon)} c \limsup_{n\to\infty}\|U_h(n)\|_1.
\]
By part (C) of Lemma~\ref{lemma.Uasy}, the righthand side is deterministic and finite, so the upper bound in part (B) has been established.

\subsection{Proof of zero liminf and average in case (B)}
We start by recalling a result of which may be found in e.g., Rugh~\cite{Rugh:1996}.
\begin{lemma} \label{lemma.rugh}
Let $C$ be a $d\times d$ real matrix. If all the eigenvalues of $C$ lie within the unit disc in the complex plane, then there exists a positive definite 
$d\times d$ real matrix $M$ such that
\[
C^T M C-M=-I_d.
\]
\end{lemma}
Conversely, the existence of a positive definite $M$ implies that all the eigenvalues of $C$ lie inside the unit disc in the complex plane.

We will have achieved our goal once we have shown the following result.
\begin{lemma}
Suppose that the matrix $A$ obeys \eqref{eq.linstab} and that there exists $\epsilon'>0$ such that  $S_h(\epsilon)$ defined by \eqref{def.Seps} obeys  $S_h(\epsilon)<+\infty$ for all $\epsilon>\epsilon'$ and $S_h(\epsilon)=+\infty$ for all $\epsilon<\epsilon'$. Then 
\[
\liminf_{n\to\infty} \|Y_h(n)\|=0, \quad \lim_{n\to\infty} \frac{1}{n}\sum_{j=1}^n \|Y_h(j)\|^2=0, \quad \text{a.s.}
\]
\end{lemma}
\begin{proof}
It has been shown above that all the eigenvalues of the matrix $C=C(h)$ lie inside the unit disc in the complex plane. Therefore, by Lemma~\ref{lemma.rugh}
there exists a positive definite matrix $M=M(h)$ such that 
\[
C(h)^T M(h) C(h)-M(h)=-I_d.
\]
Hereinafter, we write $M=M(h)$ and $C=C(h)$. 

Define the function $V:\mathbb{R}^d\to \mathbb{R}$ by $V(x)=x^TMx$ for $x\in\mathbb{R}^d$. 
We have that $Y_h(n+1)=CY_h(n)+U_h(n+1)$ for $n\geq 0$ with $Y_h(0)=\zeta$. Therefore, we have 
\begin{multline*}
V(Y_n(n+1))-V(Y_h(n))=-Y^T(n)Y(n)+Y_h(n)^TC^TMU_h(n+1)\\+U_h(n+1)^TMCY_h(n)+U_h(n+1)^TMU_h(n+1), \quad n\geq 0.
\end{multline*}
using $C^T M C-M=-I_d$ to simplify the first term on the right hand side. We now simplify the other terms on the right hand side. 
 
Since $M$ is a positive definite matrix, there exists a matrix $P$ such that $M=PP^T$. Then 
\begin{align*}
U_h(n+1)^TMU_h(n+1)&=U_h(n+1)^TPP^TU_h(n+1)\\
&=(P^TU_h(n+1))^T P^TU_h(n+1)=\|P^T U_h(n+1)\|^2_2.
\end{align*}
Define $k(n+1)=Y_h(n)^TC^TMU_h(n+1)+U_h(n+1)^TMCY_h(n)$ for $n\geq 0$. Then using the fact that $M$ is symmetric and the definition of $U_h$, we get 
\begin{align*}
k(n+1)&= (M^TC Y_h(n))^T U_h(n+1)+U_h(n+1)^T MCY_h(n)\\
&= (MC Y_h(n))^T U_h(n+1)+U_h(n+1)^T MCY_h(n) \\
&= 2\langle MCY_h(n),U_h(n+1)\rangle\\
&= 2\sqrt{h}\langle  MCY_h(n),\sigma_h(n)\xi(n+1)\rangle.
\end{align*}
Therefore 
\begin{equation}\label{eq.smallk}
k(n+1)= 2\sqrt{h} \sum_{j=1}^r \left(\sum_{i=1}^d [MCY_h(n)]_i [\sigma_h(n)]_{ij} \right) \xi_j(n+1), \quad n\geq 0.
\end{equation}
Hence we have 
\[
V(Y_h(n+1))-V(Y_h(n))=-Y_h^T(n)Y_h(n)+k(n+1)+\|P^T U_h(n+1)\|^2_2, \quad n\geq 0,
\]
so if we define 
\[
K(n)=\sum_{l=1}^{n}k(l) 
=\sum_{j=1}^r \sum_{l=1}^n \left(\sum_{i=1}^d 2\sqrt{h}[MCY_h(l-1)]_i [\sigma_h(l-1)]_{ij} \right) \xi_j(l), \quad n\geq 1,
\]
then $K$ is a martingale and 
\begin{equation}\label{eq.VmasterYn}
V(Y_h(n))-V(\zeta)
=-\sum_{l=0}^{n-1} \|Y(l)\|^2_2 + K(n)
+\sum_{l=0}^{n-1}\|P^T U_h(l+1)\|^2_2,\quad n\geq 1.
\end{equation}
We now estimate the asymptotic behaviour of the last two terms on the righthand side of \eqref{eq.VmasterYn}. The quadratic variation of $K$ is given by 
\[
\langle K \rangle(n)= \sum_{l=1}^n \left(\sum_{i=1}^d 2\sqrt{h}[MCY_h(l-1)]_i [\sigma_h(l-1)]_{ij} \right)^2. 
\]
By the Cauchy--Schwartz inequality, we have 
\begin{align*}
\langle K \rangle(n)
&\leq  \sum_{l=1}^n 4h\left(\sum_{i=1}^d [MCY_h(l-1)]_i^2 \sum_{i=1}^d [\sigma_h(l-1)]^2_{ij} \right)\\
&\leq  \sum_{l=1}^n 4h \|MCY_h(l-1)\|^2 \|\sigma_h(l-1)\|^2_F.
\end{align*}
Since $\|Y_h(n)\|$ is bounded and $\|\sigma_h(n)\to 0$ as $n\to\infty$ (by Lemma~\ref{lemma.sigmahxiLLN}) we have that 
\[
\lim_{n\to\infty} \frac{1}{n}\langle K\rangle(n)=0, \quad\text{a.s.}
\]
Arguing as in the proof of Lemma~\ref{lemma.liminfXdisc}, we see that 
\begin{equation} \label{eq.Knnto0}
\lim_{n\to\infty} \frac{1}{n}K(n)=0, \quad\text{a.s.}
\end{equation}
As for the last term on the right hand side of \eqref{eq.VmasterYn}
\[
0\leq \limsup_{n\to\infty}
\frac{1}{n}\sum_{l=0}^{n-1}\|P^T U_h(l+1)\|^2_2 
\leq \|P^T\|^2_2\limsup_{n\to\infty}
\frac{1}{n} \sum_{l=0}^{n-1} \|U_h(l+1)\|^2_2 =0, \quad\text{a.s.}
\]
by  \eqref{eq.sigmahxiLLN} in Lemma~\ref{lemma.sigmahxiLLN}. Hence
\begin{equation}\label{eq.PUnto0}
\lim_{n\to\infty}
\frac{1}{n}\sum_{l=0}^{n-1}\|P^T U_h(l+1)\|^2_2 =0, \quad\text{a.s.}
\end{equation}
Since $\|Y_h(n)\|$ is a.s. bounded, we have $V(Y_h(n))/n\to 0$ as $n\to\infty$ a.s.
Therefore, using this limit and \eqref{eq.PUnto0} and \eqref{eq.Knnto0} in \eqref{eq.VmasterYn} we get 
\begin{equation*} 
\lim_{n\to\infty} \frac{1}{n}\sum_{l=0}^{n-1} \|Y_h(l)\|^2_2=0, \quad\text{a.s.}
\end{equation*}
This proves the second statement. It also implies that $\liminf_{n\to\infty} \|Y_h(n)\|_2=0$ a.s. for otherwise we would have 
\[
\liminf_{n\to\infty} \frac{1}{n}\sum_{l=0}^{n-1} \|Y_h(l)\|^2_2>0 \quad\text{with positive probability},
\]
a contradiction.
\end{proof}

\end{document}